\documentclass[a4paper,english,bibliography=totoc,DIV=12]{scrartcl}
\def\documenttitle{Randomization for Markov chains\\
with applications to  networks in a random environment}
\def\documenttitlepdf{Randomization for Markov chains
with applications to  networks in a random environment}
\def\shorttitle{Randomized stochastic networks}

\usepackage[unicode=true,
 bookmarks=true,bookmarksnumbered=false,bookmarksopen=false,
 breaklinks=true,pdfborder={0 0 0},backref=false,colorlinks=true]
 {hyperref}

\hypersetup{
 linkcolor=blue, citecolor=blue, urlcolor=blue,  filecolor=blue,
 pdftitle={\documenttitlepdf},
 pdfauthor={Ruslan Krenzler, Hans Daduna, Sonja Otten},
 pdfkeywords={randomized random walks}{Jackson networks}{processes in random environment}%
{skipping}{reflection}{product form steady-state distribution}%
{breakdown of nodes, degrading service}%
{speed-up of service},
 pdfsubject={queueing systems}
}

\usepackage{scrpage2} 

\usepackage{tikz}
\usepackage{float}
\usepackage{caption}
\usepackage{subcaption}
\usepackage{amsmath, amssymb}
\usepackage{amsthm}

\usepackage{prettyref}

\setheadwidth{textwithmarginpar}
\setheadsepline{.4pt} 

\newcommand{\N}{\mbox{$\mathbb N$}}
\newcommand{{\Z}}{{\bf Z}}

\newtheorem{theorem}{Theorem}
\newtheorem{cor}[theorem]{Corollary}
\newtheorem{defn}[theorem]{Definition}

\newtheorem{prop}[theorem]{Proposition}
\newtheorem{ex}[theorem]{Example}

\numberwithin{theorem}{section}
\numberwithin{equation}{section}

\newcommand{\cF}{{\cal F}}

\newcommand{\X}{{\bf X}}
\newcommand{\Y}{{\bf Y}}

\newcommand{\evect}{\mathbf{e}}

\newcommand{\Jset}{\overline{J}}

\newcommand{\Fset}{\overline{F}}

\newcommand{\myV}{V}

\newcommand{\qsep}{,}

\newcommand{\groutmx}{r}

\newcommand{\routmx}{r}

\newcommand{\envmx}{R}

\newcommand{\tnfactor}{\beta}

\newcommand{\srfactor}{\gamma}

\newcommand{\srfactorvect}{\boldsymbol{\srfactor}}

\newcommand{\avvect}{{\boldsymbol{\alpha}}}

\newcommand{\nvect}{\mathbf{n}}

\newcommand{\JsetB}[1]{B(#1)}

\newcommand{\JsetW}[1]{W(#1)}

\newrefformat{thm}{Theorem \ref{#1}}
\newrefformat{cor}{Corollary \ref{#1}}
\newrefformat{ex}{Example \ref{#1}}
\newrefformat{prop}{Proposition \ref{#1}}
\newrefformat{prob}{Problem \ref{#1}}
\newrefformat{part}{Part \ref{#1}}
\newrefformat{sect}{Section \ref{#1}}
\newrefformat{def}{Definition \ref{#1}}
\newrefformat{rem}{Remark \ref{#1}}

\usepackage{pgfplots}
\usepackage{pgf}
\usetikzlibrary{shapes,arrows} 

\usepackage{ifthen}

\makeatletter
\pgfkeys{/tikz/queue head/.initial=east}
\pgfkeys{/tikz/queue size/.initial=infinite}
\pgfkeys{/tikz/queue ellipsis/.initial=false}

\def\queue@width{
	\pgf@x=\wd\pgfnodeparttextbox
	\pgfmathsetlength{\pgf@xa}{\pgfkeysvalueof{/pgf/minimum width}}%
	\ifdim\pgf@x<\pgf@xa
		\pgf@x=\pgf@xa
	\fi
}

\def\queue@height{
	\pgf@x=\ht\pgfnodeparttextbox
	\pgfmathsetlength{\pgf@xa}{\pgfkeysvalueof{/pgf/minimum height}}%
	\ifdim\pgf@x<\pgf@xa
		\pgf@x=\pgf@xa
	\fi
}

\pgfdeclareshape{queue-base}
{
	\inheritsavedanchors[from={rectangle}]
	\inheritanchorborder[from={rectangle}]
	\inheritanchor[from={rectangle}]{center}
	\inheritanchor[from={rectangle}]{north}
	\inheritanchor[from={rectangle}]{east}
	\inheritanchor[from={rectangle}]{south}
	\inheritanchor[from={rectangle}]{west}
	
	\savedmacro{\queuehead}{\pgfkeysvalueof{/tikz/queue head}}

	\saveddimen{\halfwidth}{
		\queue@width
		\divide \pgf@x by 2
	}
	\saveddimen{\halfheight}{
		\queue@height
		\divide \pgf@x by 2
	}

	\anchor{head}{
		\ifthenelse{\equal{\pgfkeysvalueof{/tikz/queue head}}{east}}{
				\northeast
				\advance \pgf@y by -\halfheight
		}{}
		\ifthenelse{\equal{\pgfkeysvalueof{/tikz/queue head}}{west}}{
				\southwest
				\advance \pgf@y by \halfheight
		}{}
		\ifthenelse{\equal{\pgfkeysvalueof{/tikz/queue head}}{north}}{
				\northeast
				\advance \pgf@x by -\halfwidth
		}{}
		\ifthenelse{\equal{\pgfkeysvalueof{/tikz/queue head}}{south}}{
				\southwest
				\advance \pgf@x by \halfwidth
		}{}
	}
	
	\anchor{tail}{
		\ifthenelse{\equal{\pgfkeysvalueof{/tikz/queue head}}{east}}{
				\southwest
				\advance \pgf@y by \halfheight
		}{}
		\ifthenelse{\equal{\pgfkeysvalueof{/tikz/queue head}}{west}}{
				\northeast
				\advance \pgf@y by -\halfheight
		}{}
		\ifthenelse{\equal{\pgfkeysvalueof{/tikz/queue head}}{north}}{
				\southwest
				\advance \pgf@x by \halfwidth
		}{}
		\ifthenelse{\equal{\pgfkeysvalueof{/tikz/queue head}}{south}}{
				\northeast
				\advance \pgf@x by -\halfwidth
		}{}
	}
}

\pgfdeclareshape{queue}
{
	\inheritsavedanchors[from={queue-base}]
	\inheritanchorborder[from={queue-base}]
	\inheritanchor[from={queue-base}]{center}
	\inheritanchor[from={queue-base}]{north}
	\inheritanchor[from={queue-base}]{east}
	\inheritanchor[from={queue-base}]{south}
	\inheritanchor[from={queue-base}]{west}
	\inheritanchor[from={queue-base}]{head}
	\inheritanchor[from={queue-base}]{tail}

	\savedmacro{\queuehead}{\pgfkeysvalueof{/tikz/queue head}}

	\saveddimen{\width}{
		\queue@width
	}

	\saveddimen{\height}{
		\queue@height
	}

	\backgroundpath{
		\pgfextractx\pgf@xa{\southwest}
		\pgfextracty\pgf@ya{\southwest}
		\pgfextractx\pgf@xb{\northeast}
		\pgfextracty\pgf@yb{\northeast}
		
 		\newboolean{vertical}
		\ifthenelse{\equal{\pgfkeysvalueof{/tikz/queue head}}{north}
		 	\OR  
			\equal{\pgfkeysvalueof{/tikz/queue head}}{south}}{
			\setboolean{vertical}{true}	
		}{
			\setboolean{vertical}{false}
		}
		
		\newboolean{ellipsis}
		\ifthenelse{\equal{\pgfkeysvalueof{/tikz/queue ellipsis}}{true}}{
			\setboolean{ellipsis}{true}	
		}{
			\setboolean{ellipsis}{false}
		}
				
		\newboolean{sizeisnumber}
		\ifthenelse{\equal{\pgfkeysvalueof{/tikz/queue size}}{infinite}
		\OR
		\equal{\pgfkeysvalueof{/tikz/queue size}}{bounded}}
		{
			\setboolean{sizeisnumber}{false}
		}{
			\setboolean{sizeisnumber}{true}
		}
		
		\ifthenelse{\boolean{vertical}}{
			\pgfpathmoveto{\pgfpoint{\pgf@xa}{\pgf@ya}}
			\pgfpathlineto{\pgfpoint{\pgf@xa}{\pgf@yb}}
			\pgfpathmoveto{\pgfpoint{\pgf@xb}{\pgf@ya}}
			\pgfpathlineto{\pgfpoint{\pgf@xb}{\pgf@yb}}	
		}{
			\pgfpathmoveto{\pgfpoint{\pgf@xa}{\pgf@ya}}
			\pgfpathlineto{\pgfpoint{\pgf@xb}{\pgf@ya}}
			\pgfpathmoveto{\pgfpoint{\pgf@xa}{\pgf@yb}}
			\pgfpathlineto{\pgfpoint{\pgf@xb}{\pgf@yb}}		
		}
		
		\newdimen \midy
		\midy=\pgf@ya
		\advance \midy by \pgf@yb
		\divide \midy by 2
		
		\newdimen \midx
		\midx=\pgf@xa
		\advance \midx by \pgf@xb
		\divide \midx by 2
		
		\ifthenelse{\boolean{vertical}}{
			\newdimen \cellheight
			\cellheight=\height

   	 		\ifthenelse{\boolean{sizeisnumber}}
			{
				\pgfmathparse{\pgfkeysvalueof{/tikz/queue size}}
				\divide\cellheight by \pgfmathresult
			}{
				\divide\cellheight by 6
			}
			
			\newdimen \currcelly
			\currcelly=\pgf@ya

			\ifthenelse{\equal{\pgfkeysvalueof{/tikz/queue size}}{infinite}
			 		\AND
			 		\equal{\pgfkeysvalueof{/tikz/queue head}}{north}}
        	{
       	 		\foreach \y in {2,3,4,5,6}{
					\advance \currcelly by \y\cellheight
					\pgfpathmoveto{\pgfpoint{\pgf@xa}{\currcelly}}
					\pgfpathlineto{\pgfpoint{\pgf@xb}{\currcelly}}
				}
				\ifthenelse{\boolean{ellipsis}}{
					\pgf@yc = \pgf@ya
					\advance \pgf@yc by \cellheight
					\filldraw (\midx, \pgf@yc) circle (1pt);
					\filldraw (\midx, \pgf@yc+5pt) circle (1pt);
					\filldraw (\midx, \pgf@yc-5pt) circle (1pt);
				}{}
       	 	}{}

			 \ifthenelse{\equal{\pgfkeysvalueof{/tikz/queue size}}{infinite}
			 		\AND
			 		\equal{\pgfkeysvalueof{/tikz/queue head}}{south}}
       	 	{
       	 		\foreach \y in {0,1,2,3,4}{
					\advance \currcelly by \y\cellheight
					\pgfpathmoveto{\pgfpoint{\pgf@xa}{\currcelly}}
					\pgfpathlineto{\pgfpoint{\pgf@xb}{\currcelly}}
				}
				\ifthenelse{\boolean{ellipsis}}{
					\pgf@yc = \pgf@yb
					\advance \pgf@yc by -\cellheight
					\filldraw (\midx, \pgf@yc) circle (1pt);
					\filldraw (\midx, \pgf@yc+5pt) circle (1pt);
					\filldraw (\midx, \pgf@yc-5pt) circle (1pt);
				}{}
        	}{}
        	
        	\ifthenelse{\equal{\pgfkeysvalueof{/tikz/queue size}}{bounded}}
       	 	{
       	 		\foreach \y in {0,1,2,4,5,6}{
					\advance \currcelly by \y\cellheight
					\pgfpathmoveto{\pgfpoint{\pgf@xa}{\currcelly}}
					\pgfpathlineto{\pgfpoint{\pgf@xb}{\currcelly}}
				}
				\pgf@yc = \pgf@ya
				\advance \pgf@yc by 3\cellheight
				\filldraw (\midx, \pgf@yc) circle (1pt);
				\filldraw (\midx, \pgf@yc+5pt) circle (1pt);
				\filldraw (\midx, \pgf@yc-5pt) circle (1pt);
        	}{}
        	
        	\ifthenelse{\boolean{sizeisnumber}}{
        		\pgfmathparse{\pgfkeysvalueof{/tikz/queue size}}
        		\foreach \y in {0,...,\pgfmathresult}{
					\advance \currcelly by \y\cellheight
					\pgfpathmoveto{\pgfpoint{\pgf@xa}{\currcelly}}
					\pgfpathlineto{\pgfpoint{\pgf@xb}{\currcelly}}
				}
        	}{}
		}{
			\newdimen \cellwidth
			\cellwidth=\width
			
   	 		\ifthenelse{\boolean{sizeisnumber}}
			{
				\pgfmathparse{\pgfkeysvalueof{/tikz/queue size}}
				\divide\cellwidth by \pgfmathresult
			}{
				\divide\cellwidth by 6
			}
			
			\newdimen \currcellx
			\currcellx=\pgf@xa
			\ifthenelse{\equal{\pgfkeysvalueof{/tikz/queue size}}{infinite}
			 		\AND
			 		\equal{\pgfkeysvalueof{/tikz/queue head}}{east}}
        	{
       	 		\foreach \x in {2,3,4,5,6}{
					\advance \currcellx by \x\cellwidth
					\pgfpathmoveto{\pgfpoint{\currcellx}{\pgf@ya}}
					\pgfpathlineto{\pgfpoint{\currcellx}{\pgf@yb}}
				}
				\ifthenelse{\boolean{ellipsis}}{
					\pgf@xc = \pgf@xa
					\advance \pgf@xc by \cellwidth
					\filldraw (\pgf@xc,\midy) circle (1pt);
					\filldraw (\pgf@xc+5pt,\midy) circle (1pt);
					\filldraw (\pgf@xc-5pt,\midy) circle (1pt);
				}{}
       	 	}{}

			 \ifthenelse{\equal{\pgfkeysvalueof{/tikz/queue size}}{infinite}
			 		\AND
			 		\equal{\pgfkeysvalueof{/tikz/queue head}}{west}}
       	 	{
	        	\foreach \x in {0,1,2,3,4}{
					\advance \currcellx by \x\cellwidth
					\pgfpathmoveto{\pgfpoint{\currcellx}{\pgf@ya}}
					\pgfpathlineto{\pgfpoint{\currcellx}{\pgf@yb}}
				}
				\ifthenelse{\boolean{ellipsis}}{
					\pgf@xc = \pgf@xb
					\advance \pgf@xc by -\cellwidth
					\filldraw (\pgf@xc,\midy) circle (1pt);
					\filldraw (\pgf@xc+5pt,\midy) circle (1pt);
					\filldraw (\pgf@xc-5pt,\midy) circle (1pt);
				}{}
        	}{}
        	
        	\ifthenelse{\equal{\pgfkeysvalueof{/tikz/queue size}}{bounded}}
       	 	{
	        	\foreach \x in {0,1,2,4,5,6}{
					\advance \currcellx by \x\cellwidth
					\pgfpathmoveto{\pgfpoint{\currcellx}{\pgf@ya}}
					\pgfpathlineto{\pgfpoint{\currcellx}{\pgf@yb}}
				}
				\pgf@xc = \pgf@xa
				\advance \pgf@xc by 3\cellwidth
				\filldraw (\pgf@xc,\midy) circle (1pt);
				\filldraw (\pgf@xc+5pt,\midy) circle (1pt);
				\filldraw (\pgf@xc-5pt,\midy) circle (1pt);
        	}{}
        	
        	\ifthenelse{\boolean{sizeisnumber}}{
        		\pgfmathparse{\pgfkeysvalueof{/tikz/queue size}}
        		\foreach \x in {0,...,\pgfmathresult}{
					\advance \currcellx by \x\cellwidth
					\pgfpathmoveto{\pgfpoint{\currcellx}{\pgf@ya}}
					\pgfpathlineto{\pgfpoint{\currcellx}{\pgf@yb}}
				}
        	}{}
        } 
	}
}

\colorlet{down.partially.bg.color}{red!50}
\definecolor{external.fg.color}{RGB}{0,0,0}
\definecolor{external.fg.color}{RGB}{0,0,0}
\newcommand{\customerslinewidth}{1pt}

\begin{document}
\title{\documenttitle}
\author{Ruslan Krenzler
\thanks{ruslan.krenzler@uni-hamburg.de}
\and Hans Daduna%
\thanks{daduna@math.uni-hamburg.de}
\and Sonja Otten%
\thanks{sonja.otten@uni-hamburg.de}
\and 
\small{Fachbereich Mathematik, Universi\"{a}t Hamburg, Bundesstra{\ss}e 55, 20146 Hamburg}
}

\maketitle
\vspace{-1cm}
\begin{abstract}
We\footnote{Preprint-No. 2014-02, Fachbereich Mathematik,
Mathematische Statistik und Stochastische Prozesse}
develop randomized modifications of Markov chains and apply these modifications to the routing  chains of customers in Jacksonian stochastic networks. The aim of our investigations is to find new rerouting schemes for non standard Jackson networks which hitherto resist computing explicitly the stationary distribution.

The non standard properties we can handle by suitable algorithms encompass several modifications of Jackson networks known in the literature, especially breakdown and repair of nodes with access modification for customers to down nodes, finite buffers with control of buffer overflow. The rerouting schemes available in the literature for these situations are special cases of our rerouting schemes, which can deal also with partial degrading of service capacities and even with speed up of service.  

In any case we require our algorithms to react on such general changes in the network with the aim to maintain the utilization of the nodes. To hold this invariant under change of service speeds (intensities) our algorithms not only adapt the routing probabilities but decrease automatically the overall arrival rate to the network if necessary. 
 
Our main application is for stochastic networks in a random environment. The impact of the environment on the network is by changing service speeds (by upgrading and/or degrading, breakdown, repair) and we implement the randomization algorithms to react to the changes of the environment. On the other side, customers departing from the network may enforce the environment to jump immediately. So our environment is not Markov for its own.

The main result is to compute explicitly the joint stationary distribution of the queue lengths vector and the environment which is of product form: Environment and    queue lengths vector, and the queue lengths over the network are decomposable.
\end{abstract}



\tableofcontents

\section{Introduction}\label{sect:RS-intro}
Queueing networks with product form steady state distribution have found many fields of applications,
e.g. production systems, telecommunications, and computer system modeling.
The success of this class of models and its relatives stems mainly from the simple structure of the
steady state distribution which provides access to easy performance evaluation procedures.

Starting from the work of Jackson \cite{jackson:57} various generalizations have been developed.
A special branch of research which recently has found some interest are product form models for
queueing networks in a random environment with product form steady state distributions.\\

For single service stations (in isolation) there is a long history with investigations on the behaviour of
the stations under external influences, which are subsumed under the term of an environment. Similarly
birth-death processes (as generalizations of classical $M/M/1/\infty$ queues) in a random environment
are well investigated. Most of this work resulted in complex steady state distributions,
 see  e.g., \cite{cogburn;torrez:81}, \cite{cogburn:80}, \cite{yechiali:73},
 \cite{kulkarni;yan:12}, \cite{falin:96}.
A special branch of related research is concerned with service systems under external influences
which cause the service process to break down or decreases availability of servers, see e.g.
\cite{zolotarev:66}, and for recent survey \cite{krishnamoorthy;pramod;chakravarthy:12}.
The results in these articles most often lack the elegance of Jackson's product form steady state
and the simplicity of the birth-death steady states.

Some exceptions are e.g. \cite{schwarz;sauer;daduna;kulik;szekli:06}, \cite{krishnamoorthy;narayanan:13}, and \cite{saffari;asmussen;haji:13}, where the environment of a production model (queue) is an associated inventory, \cite{sauer;daduna:03}, where the influence of the environment on the queue results in randomly occurring breakdowns of the production facility,
and \cite{krenzler;daduna:14}, where the environment encompasses e.g. the node's neighbours, their
status, etc. In these papers on queueing-environment processes
it is shown that the two-dimensional steady state distribution factorizes into the product of
the marginal
one-dimensional steady state distributions, shortly: In steady state and in the long run the steady state distribution for the queue and the environment decouples (for fixed time points).

Related work is by Yamazaki and Miyazawa \cite{yamazaki;miyazawa:95} where the environment of a queue
 is called the ''set of background states'' which determines transition rates of the queue and on the
other side is influenced by state changes of the queue. Yamazaki and Miyazawa prove a decomposition
property (which is in the spirit of decoupling as in Jackson's theorem)  {\bf not} for the queue and the
environment but for the joint {\sc (queueing/environment process)} and some supplementary variables which are
introduced for Markovisation of the process in case of non exponential service times and
non exponential holding times for the background states.

A first approach to find product form steady state distributions for Jackson networks in a random environment
seemingly was the work of Zhu \cite{zhu:94}. Economou \cite{economou:05}, Balsamo and Marin \cite{balsamo;marin:13},
and Tsitsiashvili, Osipova, Koliev, Baum \cite{tsitsiashvili;osipova;koliev;baum:02} continued the investigations
for partly different settings. The general procedure in these papers  for a network of single exponential servers is as follows
(explained in terms of Zhu's notation).

The key stochastic ingredients in the {\bf classical Jackson network} for node $i$ of the network
are an external Poisson-$\lambda_i$ arrival stream,
exponential-$\mu_i$ service times, and a Markovian routing scheme, which produces an overall arrival rate $\eta_i$
and a local marginal stationary distribution $\pi_i$ with
\begin{equation}\label{eq:RS-pi-i}
\pi_i(n) = \left(1-\frac{\eta_i}{\mu_i}\right) \left(\frac{\eta_i}{\mu_i}\right)^n,\quad n\in\N_0\,.
\end{equation}
In \cite{zhu:94} these fundamental parameters depend on the state of the external environment. If the
environment is in state $k$ the parameters are $\lambda_i(k), \mu_i(k), \eta_i(k)$, and with the
additional assumption that the utilizations $\rho_i(k):={\eta_i(k)}/{\mu_i(k)}$ do not depend on $k$,
i.e.
\begin{equation}\label{eq:RS-rho-invariant}
 \rho_i(k) = {\eta_i(k)}/{\mu_i(k)}={\eta_i}/{\mu_i} =:\rho_i\,,
\end{equation}
say, the local marginal stationary distribution $\pi_i$ is
\eqref{eq:RS-pi-i} again. Zhu and his followers do not explain how this independence of
$k$ for the utilizations emerges,
i.e. no (physical) control mechanism is described which holds ${\eta_i(k)}/{\mu_i(k)}$ invariant during changes of the environment.

 Tsitsiashvili, Osipova, Koliev, Baum \cite{tsitsiashvili;osipova;koliev;baum:02}  argue
for the invariance ${\eta_i(k)}/{\mu_i(k)}={\eta_i}/{\mu_i}$  by pointing on the fact that there exist
natural control mechanisms of technical and biological systems to react on changes of the environment, but no explicit rule to maintain the utilization is given.

The quest for such rules has on the other side a long history in related fields, especially in the
control of communications networks, where the term ''rerouting schemes'' describes the necessity to react, e.g.
to buffer overflow, broken down nodes, (partial) degrading  of transmission lines. Examples are described in
\cite{sauer;daduna:03} under the heading of ''skipping'', ''repeated service--random destination'', ''stalling''.
The first regime (skipping) is called ''jump over protocol'' by other authors, see \cite{dijk:88}.

 To put it into a more concrete example: These schemes try to mimic by stylized policies in communications networks an exchange of routing tables
or the rules for dynamic traffic allocation to paths as reaction to changes of the environment or the network's load situation,
for a discussion see \cite{nucci;schroeder;bhattacharyya;taft;diot:03}
 or \cite{gibbens;kelly;key:95}.\\

Our present investigation originates from the observation often made in stochastic networks
with blocking or with
unreliable servers, that it is possible to obtain explicit product form stationary distributions by
implementing a clever rerouting regime for customers who find at
a node, selected for his next entrance, the buffer full or the node broken down,
see e.g. \cite{sauer;daduna:03} and the literature cited there. 
A bulk of examples can be found throughout the book \cite{dijk:93} and 
in Chapters 1 and 9 of \cite{boucherie;dijk:10}.
These rerouting schemes maintain the utilization as in \eqref{eq:RS-rho-invariant} for those nodes which are not blocked,
resp. not broken down. We remark that these results are not covered by the model and results of Zhu,
where the $\eta_i(k)$ are required to be strictly positive for all nodes $i$ and
for all environment states $k$.

Similar product form network models for randomized medium access control protocols have found 
interest recently because of tractability of the model, see \cite{shah;shin:12}.\\

{\bf (I)} Our first contribution is that we complement the results of Zhu by providing physically
meaningful rerouting schemes which maintain in his setting the utilizations \eqref{eq:RS-rho-invariant}.

{\bf (II)} Our second contribution is to extend Zhu's results in a way that the mentioned results on networks with breakdowns in \cite{sauer;daduna:03} are covered and generalized.

We start as Zhu \cite{zhu:94} with a Jackson network with locally queue length $(n_i)$ dependent service intensities, which may depend
additionally on the actual state $k$ of the network's environment  $(\mu_i(n_i,k))$. The external arrival  rates are
$\lambda_i(k)$, and the overall arrival rates are $\eta_i(k)$. We will construct explicit control schemes which adapt the
routing  to the changes in the parameters,  when the environment changes from $k$ to $m$ and the
service intensities from $\mu_i(n_i,k)$ to $\mu_i(n_i,m)$. We will \underline{prove} that under these new control schemes the
respective ratios are maintained constant: ${\eta_i(k)}/{\mu_i(n_i,k)}={\eta_i}/{\mu_i(n_i)}$
will be  independent of $k$ (but not of $n_i$) as long as $\eta_i(k)>0$ holds.
Our theorems will cover the case $\eta_i(k)=0$ as well, which is in force e.g., if an unreliable node
$i$ is down and  therefore does not accept new customers.

The most important consequence will be that for Jackson networks in a random environment we obtain a product form stationary distribution, similar to \cite{zhu:94}[Theorem 1],
but under much more general assumptions. That is, the joint network-environment process
decouples and additionally the joint
queue length distribution of the network processes decouples as well.\

{\bf (III)} Our third contribution is an extension of the environment structures found in the mentioned previous literature because in our setting the network process influences the environment
process as well. We emphasize that different from the mentioned work in
\cite{zhu:94}, \cite{economou:05},  \cite{balsamo;marin:13}, and Tsitsiashvili, Osipova, Koliev, Baum \cite{tsitsiashvili;osipova;koliev;baum:02}, our environment process is not Markov for its own because changes in the
queue length processes may enforce the environment to immediate changes as well. There will be a two-way interaction between the service systems and the environment in our model.

We point out that the modifications of the random walks which we construct encompass  the rerouting schemes which are often found in the
literature and are called e.g. ''jump over protocol'' or ''skipping'' or
''blocking after service and retrial''. This implies that our results generalize those in
 \cite{zhu:94}, \cite{economou:05},  \cite{balsamo;marin:13}, and \cite{tsitsiashvili;osipova;koliev;baum:02}, and as well some of \cite{sauer;daduna:03}, which are not covered by the results of the previous papers.

{\bf (IV)} We start our presentation in Section \ref{sect:RS-randomwalks} with a detailed study of the routing chains for the selection of individual customers' itineraries
in the network and suitable modifications of these chains in terms of general Markov chains,
resp. random walks. The modification is realized
in analogy to principles occurring in MCMC algorithms by
attaching to any state of the chain a (state dependent)  ''acceptance probability'' to operate
via Bernoulli experiments on the original transition
matrix of the chain, which is considered as ''candidate-generating matrix'' (see \cite{bremaud:99}[Section 7.1]).

The jump is realized, if accepted, and the chain settles down at the selected state for the next time slot.
But other than in MCMC algorithms, ''not accepted'' in our modification means not necessarily for the
random walk to stay on at the departure state. We consider additionally to the standard one the
policy  that
from the selected, but not accepted state the chain tries again to find a next state, now with probabilities from the ''candidate-generating matrix''
determined by the row of the not accepted state. Thereafter acceptance is tested again,
and so on.

We will show that the steady state distributions of the original chain and the modified chains are intimately connected and that
we can express the new steady state easily in terms of the steady state of the old chain and the acceptance probabilities.

{\bf (V)} Although our research started with a quest for  new rerouting schemes for Jackson networks
in case of (partial) non-availability of servers, the developed schemes seem to be of interest for their own.

From an abstract point of view our modification of the original chain can be considered as a complicated change of measure
for the process distribution, see \cite{asmussen;glynn:07}[Chapter V, 1c, Example 8], which  results
in a surprisingly simple explicit change of measure for the stationary distribution.

The modification algorithms which lead to this change of measure can be distinguished according 
to the property to be local or global (as discussed in \cite{shah;shin:12}) with respect to the one-step transition graph of the  original chain.

In Section \ref{sect:RS-connection-simulation} we discuss the connections of our randomization
algorithm to MCMC algorithm and to von Neumann's acceptance-rejection scheme for sampling from a
complicated distribution, and furthermore compare the different modifications of the
random walk with respect to Peskun ordering, and the consequences thereof.\\

{\bf (VI)} The applications of the randomized random walk algorithms to stochastic networks 
in the present paper is in the spirit
of ''Performability'' as  stated in \cite{meyer:84}[p. 648]: ''If computing system performance is degradable, then \dots system evaluation must deal simultaneously with aspects of both
performance and reliability.'' A more recent compendium on these topics is \cite{haverkort;marie;rubino;trivedi:01}.
We start with a standard Jackson network and consider the situation where the service capacities
of nodes are degraded and by some network control the utilization of the nodes should be maintained.
(This is the situation of Zhu et al. in the mentioned papers as long as no node is completely broken down.)

The network controller's policy is in Section \ref{sect:RS-degraded} to reject a portion of the
load offered to the network and to redistribute the admitted load. This is organized by applying
our modification algorithms for random walks to the routing chains for the customers.
Additionally to partial degrading we can handle with our algorithm complete breakdown of nodes, and
in Section \ref{sect:RS-upgraded} we allow even to speed up service at some nodes, while others are degraded or broken down.
In any case the utilization of the nodes which are not completely down is maintained by the control policy.

While these sections are concerned with transforming a Jackson network into another one where
service capacities are degraded and/or upgraded and routing is adapted,  maintaining the nodes' utilization and therefore the
joint product form stationary distribution, we utilize the obtained transformation rules
in Section \ref{sect:RS-JN-random-env} to adapt routing by different algorithms to the impact
of a dynamically changing environment. The environment's changes cause the nodes' service
capacities to degrade and/or upgrade, even complete breakdown with following
repair or partial repair can be handled.

The most surprising result is an iterated  product form stationary distribution of the
system process, which is a multidimensional Markov process recording jointly the environment's
status and the joint queue lengths vector for the network process.
The product form which is obtained says that
\begin{enumerate}
  \item the queue length vector and the environment status, and
  \item inside of the joint queue lengths the local queue lengths
\end{enumerate}
asymptotically and in equilibrium decouple
(are decomposable in the sense of \cite{yamazaki;miyazawa:95}). This result is even more remarkable as we do not require
that the environment process is a Markov process of its own, as it is necessary for the results
proved in  \cite{zhu:94}, \cite{economou:05},  \cite{balsamo;marin:13}, and \cite{tsitsiashvili;osipova;koliev;baum:02}.
We can handle a two-way interaction between environment which enforces changes of the nodes'
service capacity when it changes, and the queueing network process, which triggers
immediate jumps of the environment, when a customer departs from the network.
Such (more complicated) two-way interaction is investigated in \cite{yamazaki;miyazawa:95}
as well, but due to the more complicated structure they loose in their results  
decoupling of the (single) queue length  and the environment status.\\

\noindent
{\bf Notation and conventions (noch zu ueberarbeiten):}
\begin{itemize}

\item $\mathbb{R}_{0}^{+}=[0,\infty)$
 , $\mathbb{R}^{+}=(0,\infty)$
 , $\mathbb{N}={1,2,3,\dots}$, $\mathbb{N}_{0}=\{0\}\cup\mathbb{N}$

\item
For sets $A,B$ we write  $A \subset B$ for $A$ which is a  subset of $B$ or equals $B$, and
we write $A \subsetneq B$ for $A$ which  is a  subset of $B$ but does not equal $B$.

\item
Throughout, the node set of our graphs (networks) are denoted by $\Jset :=\{1,\dots,J\}$,
and the ''extended node set'' is $\Jset_0 := \{0,1,\dots,J\}$, where ''$0$'' refers to the
external source and sink of the network.

\item
$\evect_j$ is the standard j-th base vector in  $\N^{\Jset}_0$ if $1\leq j\leq J$.

\item
$\nvect = (n_j:j\in \Jset)$.

\item
For any finite index set $\Fset=\{0,1,\dots,F\}$ and any $\avvect=(\alpha_j: j\in \Fset)$ we define a matrix $I_{\avvect}$ which is the diagonal matrix indexed by $\Fset$
with $\alpha_{i}$ on its diagonal, i.e.
\[
I_{\avvect}:=\left(\begin{array}{cccc}
\alpha_{0}\\
 & \alpha_{1}\\
 &  & \ddots\\
 &  &  & \alpha_{F}
\end{array}\right)\,,
\]
and similarly we define $I_{(1-\avvect)}$ indexed by $\Fset$ as
\[
I_{(1-\avvect)}:=\left(\begin{array}{cccc}
1-\alpha_{0}\\
 & 1-\alpha_{1}\\
 &  & \ddots\\
 &  &  & 1-\alpha_{F}
\end{array}\right)\,.
\]
.
\item
Here and elsewhere we agree that empty entries in a matrix are read to be zero.

\item For $\mathbf{x}=(x_j: j\in \Fset)$ we define
$||\mathbf{x}||_{\infty}:=\sup_{j \in \Fset}{|x_j|}$.

\item $1_{[expression]}$ is the indicator function which is $1$ if $expression$
is true and $0$ otherwise.

\item All random variables occurring in the sequel are defined on a common underlying probability space
$(\Omega, \cF, P)$.
\end{itemize}

\section{Randomized random walks}\label{sect:RS-randomwalks}
Let $X=(X_{n}:n\in\mathbb{N}_{0})$ be a homogeneous irreducible Markov chain on a
finite state space $\Fset$ with one-step transition probability matrix $\groutmx =(\groutmx (i,j):i,j\in \Fset)$ and
(unique) steady state distribution $\eta=(\eta_i:i\in \Fset)$.
This chain represents in our network applications the homogeneous Markov chain that describes
the random walk (routing) of the customers on the nodes of the network.
In this application scenario the routing will be modified by a network controller as a reaction to
changes of the network's parameter due to the impact of the environment.

The general principle is: The transition matrix $r$ will
be used as a ''candidate generating matrix'' for the next state of the random walk. The candidate state
will be accepted with state dependent probabilities and we develop different policies to continue when
the proposed state is rejected.

\subsection{Randomized skipping}
\label{sect:RS-randomskip}
The following modification of $X$ with prescribed $B \subsetneq \Fset$ is  known under the terms of~
''Markov chains with taboo set $B$'',  or ''jump-over-protocol for $B$'',  or ''skipping $B$''.
In the realm of queueing network theory this principle for modifying a Markov chain
seems to be  introduced independently  several times and was used  to
resolve blocking, see e.g. \cite{dijk:88}, \cite{economou;fakinos:98} and \cite{serfozo:99}[Chapter 3.6],
(where it is called blocking and rerouting)
and the references therein. As a general methodology  skipping
 was already introduced by Schassberger \cite{schassberger:84a} and  later on, based on
Schassberger's result, it was used in \cite{daduna;szekli:96} to construct general
abstract network processes.

An intuitive description of the skipping principle can be given in terms of a random walk on
$\Fset$ governed by $\groutmx$:
If the random walker's (RW) path governed by the Markov chain  is restricted by a taboo set
$B\subsetneq \Fset$, RW applies\\
{\bf Skipping $B$}:
If RW is in state $i\in \Fset$ and selects (with probability $\groutmx(i,j)$) its destination
$j\in \Fset \setminus B$, this jump is allowed and immediately performed.
If (with probability $\groutmx(i,k)$) he decides to jump to state $k\in B$, he only performs an imaginary jump to $k$, spending no time there, but jumping on immediately governed by the
matrix $\groutmx,$ i.e. with probability $\groutmx(k,l)$ he selects another possible successor state $l$; if
$l\in \Fset \setminus B$, then the jump is performed immediately, but if $l\in B$, RW has to perform another random choice as  if he would depart from $l$; and so on.\\

Our extended modification scheme for $X$ and $\groutmx$ is a randomized generalization of that skipping
scheme. 
For the states in $\Fset$ are given  ''acceptance probabilities'' by some vector
$\avvect=(\alpha_j\in [0,1]: j\in \Fset)$. The random walker (RW)  selects his itinerary under $r$ and the constraints $\avvect$ by \\
{\bf Randomized skipping with acceptance probabilities} $\avvect$:
If RW is in state $i\in \Fset$ and selects (with probability $\groutmx(i,j)$) its destination
$j\in \Fset$, a Bernoulli experiment is performed with success (acceptance)
probability $\alpha_j$,
independent of the past, given $j$. If the experiment is successful ($=1$),
this jump is accepted, immediately performed, and RW settles down at $j$ for at least one time slot.
If the experiment is not successful ($=0$), this jump is not accepted and
RW only performs an imaginary jump to $j$, spending no time there, but jumping on immediately governed by the
matrix $\groutmx,$ i.e. with probability $\groutmx(j,l)$ he selects another possible successor state $l$;
thereafter a Bernoulli experiment is performed with success (acceptance)
probability $\alpha_l$,
independent of the past, given $l$. If the experiment is successful ($=1$),
this jump is accepted, immediately performed, and RW settles down at $l$ for
at least one time slot.
If the experiment is not successful ($=0$), this jump is not accepted and
RW only performs an imaginary jump to $l$, spending no time there, but jumps on immediately according to the  matrix $r$; and so on.

\begin{ex}\label{ex:RS-skipping-deterministic}
If for $B \subsetneq \Fset$
\begin{equation*} 
\alpha_j =
\begin{cases}
0 & \text{if}~~~j\in B,\\
1 & \text{if}~~~j\in \Fset \setminus B\,,
\end{cases}
\end{equation*}
then we have exactly the classical skipping over taboo set $B$ as described above, because a jump to
$j\in B$ is never accepted, but whenever a jump to $j\in \Fset \setminus B$ is proposed, this will be accepted with probability $1$.

If we have some general $\avvect$ then the set $B(\avvect)=\{j\in \Fset: \alpha_j =0\}$
is a taboo set for the ''randomized skipping process''.
\end{ex}
\subsubsection{Transition matrix}\label{sect:RS-skippingmatrix}
It is easy to see that this construction of a modified chain by randomized skipping
generates a new homogeneous Markov chain, the
transition matrix of which will be denoted by $\groutmx^{(\avvect)}$. For simplicity of presentation we will denote a Markov chain under this regime by $X^{(\avvect)}$.

To determine the transition probabilities
$\groutmx^{(\avvect)}(i,j)$ we construct an auxiliary absorbing Markov chain $(X^{(A)},Y^{(A)})$ with state space $\Fset\times\{0,1\}$ such that
\begin{itemize}
 \item  $X^{(A)}$ records the itinerary of RW during his (possibly many) imaginary jumps until its
candidate state  is accepted - if this happened RW settles down there forever, because the chain
 $(X^{(A)},Y^{(A)})$ is absorbed,
 \item  $Y^{(A)}$ indicates whether a candidate state is accepted or not,
 \item  the states in $\Fset\times\{1\}$ are absorbing,
 \item  initial states for $(X^{(A)},Y^{(A)})$ are restricted to $\Fset\times\{0\}$, and therefore
$Y^{(A)}$  stays at $0$ until absorption of $(X^{(A)},Y^{(A)})$ in  $\Fset\times\{1\}$.
\end{itemize}
The transition probabilities for  $(X^{(A)},Y^{(A)})$ are for $i,j\in \Fset$ as long as at time $n$ the candidate state is not accepted
\begin{eqnarray}
P(X^{(A)}_{n+1}=j,Y^{(A)}_{n+1}=1|X^{(A)}_{n}=i,Y^{(A)}_{n}=0) & = & \groutmx(i,j) \alpha_{j}
      =(\groutmx\cdot I_\avvect)_{ij},\label{eq:RS-Amatrix1}\\
P(X^{(A)}_{n+1}=j,Y^{(A)}_{n+1}=0|X^{(A)}_{n}=i,Y^{(A)}_{n}=0) & = & \groutmx(i,j) (1-\alpha_{j})
      =(\groutmx\cdot I_{(1 - \avvect)})_{ij},\label{eq:RS-Amatrix2}
\end{eqnarray}
and thereafter
\begin{eqnarray}
P(X^{(A)}_{n+1}=j,Y^{(A)}_{n+1}=1|X^{(A)}_{n}=i,Y^{(A)}_{n}=1) & = & \delta_{ij},\label{eq:RS-Amatrix3}\\
P(X^{(A)}_{n+1}=j,Y^{(A)}_{n+1}=0|X^{(A)}_{n}=i,Y^{(A)}_{n}=1) & = & 0\,.\label{eq:RS-Amatrix4}
\end{eqnarray}

We denote the first entrance time into  $\Fset\times\{1\}$ of $(X^{(A)},Y^{(A)})$ by
\begin{equation}\label{eq:RS-tauA}
\tau^{(A)} := \inf \{n\geq 1 |(X^{(A)}_n,Y^{(A)}_n)\in  F\times\{1\}\}
= \inf \{n\geq 1 |Y^{(A)}_n = 1\}\,,
\end{equation}
which is $P$-a.s. finite.
\begin{theorem}
\label{thm:RS-r-alpha}
For the Markov chain $X=(X_n:n\in \N_0)$ with transition matrix $\groutmx=(\groutmx(i,j):i,j\in \Fset)$
and a non zero vector  $\avvect=(\alpha_j:j\in \Fset)$ of acceptance probabilities
denote by  $X^{(\avvect)}$ the Markov chain modification of $X$ under randomized skipping and by $\groutmx^{(\avvect)}$ the transition matrix of $X^{(\avvect)}$, and
by $B(\avvect)=\{j\in \Fset: \alpha_j =0\}\subsetneq \Fset$  the taboo set for  $X^{(\avvect)}$.\\
Then from the auxiliary chain  $(X^{(A)},Y^{(A)})$ we obtain
\begin{equation}\label{eq:RS-r-alpha1}
\groutmx^{(\avvect)}(i,j) = P\left(X^{(A)}(\tau^{(A)})=j |(X^{(A)}_0,Y^{(A)}_0)=(i,0)\right),
\quad i,j\in \Fset\,.
\end{equation}
The Markov chain $X^{(\avvect)}$ with state space $\Fset$ and transition matrix $\groutmx^{(\avvect)}$ is irreducible on\\ $\Fset\setminus B(\avvect)$, and it holds
\begin{equation}\label{eq:RS-r-alpha2}
\groutmx^{(\avvect)}=\sum_{k=0}^{\infty}\left(\groutmx I_{(1-\avvect)}\right)^{k}\groutmx I_{\avvect}
=\left(I - \groutmx I_{(1-\avvect)}\right)^{-1} \groutmx I_{\avvect}\,.
\end{equation}
\end{theorem}
\begin{proof}
Irreducibility of $X^{(\avvect)}$ on $\Fset\setminus B(\avvect)$ and the first statement \eqref{eq:RS-r-alpha1} follows directly from the construction.\\
For $i\in \Fset,j\in B(\avvect)$ we  obtain from the definition of acceptance probability
$\alpha_j =0$
\begin{equation}\label{eq:RS-Amatrix6}
\groutmx^{(\avvect)}(i,j) = 0.
\end{equation}
Note, that from \eqref{eq:RS-Amatrix2} follows
\begin{eqnarray}
&&P\left(X^{(A)}_n=k, Y^{(A)}_n=0, Y^{(A)}_{n-1}=0,\dots\right. \label{eq:RS-Amatrix5}\\
&&\left.\qquad\qquad\dots, Y^{(A)}_{2}=0, Y^{(A)}_{1}=0, |(X^{(A)}_0,Y^{(A)}_0)=(i,0)\right)
=\left( \left(\groutmx I_{(1 - \avvect)}\right)^{n}\right)_{ik}\nonumber
\end{eqnarray}
  It follows
for $i\in \Fset,j\in \Fset \setminus B(\avvect)$
\begin{eqnarray*}
&&\groutmx_{ij}^{(\avvect)}
  = P\left(X^{(A)}(\tau^{(A)})=j |(X^{(A)}_0,Y^{(A)}_0)=(i,0)\right)\\
 & = & P\left(X^{(A)}(\tau^{(A)})=j,Y^{(A)}(\tau^{(A)})=1 |(X^{(A)}_0,Y^{(A)}_0)=(i,0)\right)\\
  & = & \sum_{n=1}^{\infty}   P\left(X^{(A)}(\tau^{(A)})=j,Y^{(A)}(\tau^{(A)})=1, \tau^{(A)}=n
  |(X^{(A)}_0,Y^{(A)}_0)=(i,0)\right)\\
 & = & \sum_{n=1}^{\infty}\sum_{k\in \Fset} 
  P\left(X^{(A)}_n=j,Y^{(A)}_n=1,X^{(A)}_{n-1}=k, Y^{(A)}_{n-1}=0, Y^{(A)}_{n-2}=0,\dots\right.\\
 & & \left.\qquad\qquad\qquad\qquad \dots, Y^{(A)}_{2}=0, Y^{(A)}_{1}=0
  |(X^{(A)}_0,Y^{(A)}_0)=(i,0)\right)\\
 & \stackrel{_{(1)}}{=} & \sum_{n=1}^{\infty}\sum_{k\in \Fset}
   P\left(X^{(A)}_n=j,Y^{(A)}_n=1|X^{(A)}_{n-1}=k, Y^{(A)}_{n-1}=0\right)\\
 && \qquad\qquad \cdot P\left(X^{(A)}_{n-1}=k, Y^{(A)}_{n-1}=0, Y^{(A)}_{n-2}=0,\right.\dots\\
 && \qquad\qquad\qquad\qquad\qquad\left.\dots, Y^{(A)}_{2}=0, Y^{(A)}_{1}=0 |(X^{(A)}_0,Y^{(A)}_0)=(i,0)\right)\\
  & \stackrel{_{(2)}}{=} & \sum_{k\in \Fset} \sum_{n=1}^{\infty} \left(\groutmx I_{\avvect}\right)_{kj}
  \left( \left(\groutmx I_{(1-\avvect)}\right)^{n-1}\right)_{ik}
  \stackrel{(3)}{=}  \sum_{k\in \Fset} \left(\groutmx I_{\avvect}\right)_{kj} \left(\sum_{n=1}^{\infty}
  \left(\groutmx I_{(1-\avvect)}\right)^{n-1}\right)_{ik}\\
 & = &\left(\left(I - \groutmx I_{(1-\avvect)}\right)^{-1} \groutmx I_{\avvect}\right)_{ij}\,,
  \end{eqnarray*}
  which is \eqref{eq:RS-r-alpha2}.
Here (1) utilizes the Markov property of $(X^{(A)},Y^{(A)})$, (2) follows from \eqref{eq:RS-Amatrix5} and \eqref{eq:RS-Amatrix1},
and in (3) convergence of $\sum_{n=1}^{\infty} \left(\groutmx I_{(1-\avvect)}\right)^{n-1}$
follows from irreducibility of $\groutmx$ and from the substochasticity of $\groutmx I_{(1-\avvect)}$, which is strict for at least one row.
\end{proof}
We remark that we can see the stochasticity of $r^{(\avvect)}$ from
\begin{eqnarray*}
\groutmx^{(\avvect)}\mathbf{e} & \stackrel{\eqref{eq:RS-r-alpha2}}{=} & (I-\groutmx I_{(1-\avvect)})^{-1}
\big[\groutmx \underbrace{I_{\avvect}}_{=(I-I_{(1-\avvect)})}\big]\mathbf{e}
  =  (I-\groutmx I_{(1-\avvect)})^{-1}\left[\groutmx {(I-I_{(1-\avvect)})}\right]\mathbf{e}\\
  & = &(I-\groutmx I_{(1-\avvect)})^{-1}\big[\underbrace{\groutmx \mathbf{e}}_{= I \mathbf{e}} - \groutmx I_{(1-\avvect)}\mathbf{e}\big]
  =  (I-\groutmx I_{(1-\avvect)})^{-1}(I-\groutmx I_{(1-\avvect)})\mathbf{e}
  =  \mathbf{e}\,.
\end{eqnarray*}
Furthermore, we note that the states in $B(\avvect)$ are inessential, because the $B(\avvect)$-columns
$(\groutmx^{(\avvect)}_{ij}:i\in \Fset )$ for  $j\in B(\avvect)$ are zero.

\begin{ex}\label{ex:RS-routing-skipping}
Consider a set  $\Fset = \{0,1,2,3,4\}$
and a routing matrix
\begin{equation*}
\groutmx=\left(\begin{array}{c|ccccc}
 & 0 & 1 & 2 & 3 & 4\\
\hline 0 &  & 1\\
1 &  &  & 1\\
2 &  &  &  & 0.6 & 0.4\\
3 & 1\\
4 &  &  &  & 1
\end{array}\right)\,.
\end{equation*}

After applying the skipping rule to $r$
with availability vector $\avvect=(1,1,0.5,1,1)$
we get
\begin{equation*}
\groutmx^{(\avvect)}=\left(\begin{array}{c|ccccc}
 & 0 & 1 & 2 & 3 & 4\\
\hline 0 &  & 1\\
1 &  &  & 0.5 & 0.3 & 0.2 \\
2 &  &  &  & 0.6 & 0.4\\
3 & 1\\
4 &  &  & & 1
\end{array}\right)\,.
\end{equation*}
See Figure \ref{fig:RS-routing-skipping}.
\begin{figure}[H]
\centering
\begin{subfigure}[b]{0.3\textwidth}
\begin{tikzpicture}

\newcommand{\yslant}{0.2} 
\newcommand{\xslant}{-0.5}
\newcommand{\height}{6}

\begin{scope}


\end{scope}

\begin{scope}
  \node(0-1-orgn)[] at (2.25,0.5) {};
  \node(1-orgn)[] at (0,0.5) {};
  \node(2-orgn)[] at (2.5,2.5) {};
  \node(2-orgn-over)[] at (3,4) {};
  \node(2-orgn-under)[] at (2.5,2) {};
  \node(2-orgn-1)[] at (3,2.5) {}; 
  \node(2-orgn-2)[] at (3,3) {};
  \node(2-orgn-3)[] at (0,3) {};
  \node(3-orgn)[] at (4.5,1) {};
  \node(3-orgn-over)[] at (4.5,4) {};
  \node(4-orgn)[] at (0,2.75) {};
  \node(4-orgn-over)[] at (0,4) {};
  \node(4-orgn-under)[] at (0.5,2) {};
  \node(5-orgn)[] at (0,1.75) {};

  \node(0-1)[draw, shape=circle, color=external.fg.color] at (0-1-orgn) {$0$};
  \node(1)[draw, shape=circle] at (1-orgn) {$1$};
  \node(2-white)[draw, shape=circle, fill=white] at (2-orgn) {$2$};
  \node(2)[draw, shape=circle, fill=white] at (2-orgn) {$2$};
  \node(3)[draw, shape=circle] at (3-orgn) {$3$};
  \node(4)[draw, shape=circle] at (4-orgn) {$4$};

\draw[arrows={-latex}, thick,
  line width=\customerslinewidth] (1) -- (2.west);
\draw[arrows={-latex}, thick,
  line width=\customerslinewidth] (2) -- (3);
\draw[arrows={-latex}, thick,
  line width=\customerslinewidth] (3) -- (0-1);
\draw[arrows={-latex}, thick,
  line width=\customerslinewidth] (0-1) -- (1.east);
\draw[arrows={-latex}, thick,
  line width=\customerslinewidth] (2.west) -- (4.east);

\draw[arrows={-latex}, thick,
  line width=\customerslinewidth] (4) -- (3.west);

\end{scope}
\end{tikzpicture}
\caption{original matrix $\groutmx$}
\end{subfigure}
\hspace{2cm}
\begin{subfigure}[b]{0.3\textwidth}
\begin{tikzpicture}

\newcommand{\yslant}{0.2} 
\newcommand{\xslant}{-0.5}
\newcommand{\height}{6}

\begin{scope}


\end{scope}

\begin{scope}
  \node(0-1-orgn)[] at (2.25,0.5) {};
  \node(1-orgn)[] at (0,0.5) {};
  \node(2-orgn)[] at (2.5,2.5) {};
  \node(2-orgn-under)[] at (2.5,2) {};
  \node(2-orgn-1)[] at (3,2.5) {}; 
  \node(2-orgn-2)[] at (3,3) {};
  \node(2-orgn-3)[] at (0,3) {};
  \node(3-orgn)[] at (4.5,1) {};
  \node(3-orgn-over)[] at (4.5,4) {};
  \node(4-orgn)[] at (0,2.75) {};
  \node(4-orgn-over)[] at (0,4) {};
  \node(4-orgn-under)[] at (0.5,2) {};
  \node(5-orgn)[] at (0,1.75) {};

  \node(0-1)[draw, shape=circle, color=external.fg.color] at (0-1-orgn) {$0$};
  \node(1)[draw, shape=circle] at (1-orgn) {$1$};
  \node(2-white)[draw, shape=circle, fill=white] at (2-orgn) {$2$};
  \node(2)[draw, shape=circle, fill=down.partially.bg.color] at (2-orgn) {$2$};
  \node(3)[draw, shape=circle] at (3-orgn) {$3$};
  \node(4)[draw, shape=circle] at (4-orgn) {$4$};

\draw[arrows={-latex}, thick,
  line width=\customerslinewidth] (1) -- (2.west);
\draw[arrows={-latex}, thick,
  line width=\customerslinewidth] (2) -- (3);
\draw[arrows={-latex}, thick,
  line width=\customerslinewidth] (3) -- (0-1);
\draw[arrows={-latex}, thick,
  line width=\customerslinewidth] (0-1) -- (1.east);

\draw[arrows={-latex}, thick,
  line width=\customerslinewidth] (2.west) -- (4.east);

\draw[arrows={-latex}, thick,
  line width=\customerslinewidth] (4) -- (3.west);

\draw[arrows={-latex}, thick, line width=\customerslinewidth,color=blue] (1) .. controls (2-orgn-under) .. (3.west);

\draw[arrows={-latex}, thick, line width=\customerslinewidth,
color=blue,bend right=40] (1) .. controls (4-orgn-under) ..(4);
\end{scope}
\end{tikzpicture}
\caption{modified matrix $\groutmx^{(\avvect)}$}
\end{subfigure}
\caption{Matrix $\groutmx$ on $\Fset = \{0,1,2,3,4\}$ and $\groutmx^{(\avvect)}$ with $\avvect=(1,1,0.5,1,1)$
from \prettyref{ex:RS-routing-skipping}.
\label{fig:RS-routing-skipping}
}
\end{figure}
\end{ex}

\subsubsection{Stationary distribution}\label{sect:RS-skippingequilibrium}
Recall that $X$ is irreducible with transition matrix $\groutmx$ on the finite state space $\Fset$ and has the unique
stationary distribution  $\eta=(\eta_j:j\in \Fset)$.
For a non zero vector $\avvect=(\alpha_j:j\in \Fset)$ the chain $X^{(\avvect)}$ with state space $\Fset$
and  transition matrix $\groutmx^{(\avvect)}$
is irreducible on $\Fset\setminus B(\avvect)$, with $B(\avvect)=\{j\in \Fset:\alpha_j=0\}$.
We denote its stationary distribution by $\eta^{(\avvect)}=(\eta^{(\avvect)}(j):j\in \Fset)$,
which has support $\Fset\setminus B(\avvect)$. We will study the relation between $\eta$ and
 $\eta^{(\avvect)}$.

\begin{prop}
\label{prop:GS-y-from-eta}Let $x$ be a solution of the balance equation
$x\cdot \groutmx=x$ then
\[
y=x\cdot I_{\avvect}
\]
 solves the steady state equation $y\cdot \groutmx^{(\avvect)}=y$ of the modified Markov chain
 with randomized skipping.
\end{prop}
\begin{proof} From  $x\cdot \groutmx  =  x$ we obtain
\begin{eqnarray*}
 x\groutmx\underbrace{\left(I_{\avvect}+I_{(1-\avvect)}\right)}_{=I}  =  x
\Longleftrightarrow x\groutmx I_{\avvect}  =  x-x\groutmx I_{1-\avvect}
\Longleftrightarrow x\groutmx I_{\avvect}  =  x(I-\groutmx I_{(1-\avvect)})\\
\Longleftrightarrow x\underbrace{(I-\groutmx I_{(1-\avvect)})(I-\groutmx I_{(1-\avvect)})^{-1}}_{=I}\groutmx I_{\avvect}  =  x(I-\groutmx I_{(1-\avvect)})
\end{eqnarray*}
and with $y := x(I-\groutmx I_{(1-\alpha)})$ follows that a required solution of
$y\cdot \groutmx^{(\avvect)}  = y$ is
\begin{eqnarray*}
y & 
{=} & x(I-\groutmx I_{(1-\avvect)})=x-xI_{(1-\avvect)}=xI_{\avvect}
\,.
\end{eqnarray*}
\end{proof}

\begin{prop}\label{prop:RS-x-from-eta-alpha}
Let $y$ be a solution of the balance
equation $y\cdot \groutmx^{(\avvect)}=y$ then 
\begin{equation}
x=y(I-\groutmx I_{(1-\avvect)})^{-1}\label{eq:RS-x-definition}
\end{equation}
 is a solution of the steady state  equation    
$x\cdot \groutmx=x$ 
 and it holds
\begin{equation}
y=(\alpha_{j}x_{j}:j\in \Fset)\,.\label{prop:GS-eta-alpha-from-eta}
\end{equation}
\end{prop}
\begin{proof}
From the definition of  $\groutmx^{(\avvect)}$  follows
\begin{eqnarray*}
y\underbrace{(I-\groutmx I_{(1-\avvect)})^{-1}\groutmx I_{\avvect}}_{=\groutmx^{(\avvect)}} & = & y\underbrace{(I-\groutmx I_{(1-\avvect)})^{-1}(I-\groutmx I_{(1-\avvect)})}_{=I}
\end{eqnarray*}
So $x=y(I-\groutmx I_{(1-\avvect)})^{-1}$ fulfils
\begin{eqnarray*}
x\groutmx I_{\avvect}  =  x(I-\groutmx I_{(1-\avvect)})
\Longleftrightarrow x\groutmx\underbrace{\left(I_{\avvect}+I_{(1-\avvect)}\right)}_{=I}  =  x
\end{eqnarray*}
which is $x\cdot \groutmx  =  x$. \eqref{prop:GS-eta-alpha-from-eta} follows
from $x(I-\groutmx I_{(1-\avvect)})  =  y$ as in  \prettyref{prop:GS-y-from-eta}.
\end{proof}

\begin{theorem}\label{thm:RS-eta-alpha}
If $\eta$ is the unique steady state distribution of the irreducible Markov chain $X$ on the finite state space $\Fset$, then the  unique steady state distribution $\eta^{(\avvect)}$ of
the Markov chain $X^{(\avvect)}$ is
\begin{equation}
\eta^{(\avvect)}=\left(C^{(\avvect)}\right)^{-1}(\eta_{j}\alpha_{j}:j\in \Fset)
\label{eq:RS-eta-alpha-from-eta}
\end{equation}
with normalization constant
\begin{equation}
C^{(\avvect)}=(\eta I_{\avvect}\mathbf{e}) = \langle\eta,\avvect\rangle\label{eq:RS-eta-alpha-normalization-constant}\,.
\end{equation}
$\eta^{(\avvect)}$ has support $\Fset\setminus B(\avvect)$.
\end{theorem}
\begin{proof}
From Propositions \ref{prop:GS-y-from-eta} and \ref{prop:RS-x-from-eta-alpha} we know that there is a one-to-one connection between all solutions of $x\cdot \groutmx  =  x$ and
$y\cdot \groutmx^{(\avvect)}=y$.\\
 From the assumptions on $X$ we know that $\eta$ is the
unique stochastic solution of $x\cdot \groutmx  =  x$.
Its normalized companion is therefore \eqref{eq:RS-eta-alpha-from-eta} by \eqref{prop:GS-eta-alpha-from-eta}.
\end{proof}

\subsection{Randomized reflection}
\label{sect:RS-randomreflec}
An important problem in the control of transmission networks is to react to full buffers at receiver
stations by the network provider. There are many detailed control regimes to resolve blocking
which occurs when buffers overflow. It turned out that it is usually difficult to construct
analytical network models for this problem which admit closed form solutions for main performance metrics.

The following control principle is common to resolve blocking situations in networks with blocking of stations due to full buffers or blocking due to  resource sharing, and is
called blocking principle {\sc Repetitive Service -
  Random Destination (rs--rd)}. \\
For applications in modeling of communication protocols in systems with finite buffers or for ALOHA-type protocols see \cite[Section 5.11]{kleinrock:76},
Within the abstract framework of reversible processes it occurs in \cite[Proposition II.5.10]{liggett:85}.

The fundamental idea is that whenever a packet is to be send from some node $i$ to node $j$ and it
is observed that the buffer for incoming packets at $j$ is full the packet is rejected (lost)
and node $i$
who has saved a copy tries to resend this packet ({\sc Repetitive Service)}, but not necessarily to $j$ but selects another destination node $k$ ({\sc Random Destination}). This reflection procedure is iterated until the packet is sent to a node with free buffer places.\\

We will apply this scheme to modify the homogeneous irreducible Markov chain
$X=(X_{n}:n\in\mathbb{N}_{0})$  on the
finite state space $\Fset$ with one-step transition probability matrix $\groutmx=(\groutmx(i,j):i,j\in \Fset)$ and
(unique) steady state distribution $\eta=(\eta_i:i\in \Fset)$.

In terms of a random walk on
$\Fset$ governed by $r$ an intuitive description is:
If the random walker's (RW) path governed by the Markov chain  is restricted by a taboo set
$B\subsetneq \Fset$, RW applies\\
{\bf Reflection at $B$}:
If RW is in state $i\in \Fset$ and selects (with probability $\groutmx(i,j)$) its destination
$j\in \Fset \setminus B$, this jump is allowed and immediately performed.
If (with probability $\groutmx(i,k)$) he decides to jump to state $k\in B$, he is reflected at $k$ and spends at least one further time slot at node $i$. Thereafter he restarts the procedure possibly with
another destination node, i.e. with probability $\groutmx(i,l)$ he selects successor state $l$; if
$l\in \Fset\setminus B$ then the jump is performed immediately, but if $l\in B$ RW is reflected again; and so on.\\

Our extended modification scheme for $X$ and $\groutmx$ is a randomized generalization of that reflection
scheme.
For the states in $\Fset$ are given  ''acceptance probabilities'' by some vector
$\avvect=(\alpha_j\in [0,1]: j\in \Fset)$. The random walker (RW)  selects his itinerary under $r$ and the constraints $\avvect$ by \\
{\bf Randomized reflection with acceptance probabilities} $\avvect$:
If RW is in state $i\in \Fset$ and selects (with probability $\groutmx(i,j)$) its destination
$j\in \Fset$, a Bernoulli experiment is performed with success (acceptance)
probability $\alpha_j$,
independent of the past, given $j$. If the experiment is successful ($=1$),
this jump is accepted, immediately performed, and RW settles down at $j$ for at least one time slot.
If the experiment is not successful ($=0$), this jump is not accepted and
RW stays on at $i$ for at least one further time slot. If this slot expires with probability $\groutmx(i,l)$ RW selects another possible successor state $l$;
thereafter a Bernoulli experiment is performed with success (acceptance)
probability $\alpha_l$,
independent of the past, given $l$. If the experiment is successful ($=1$),
this jump is accepted, immediately performed, and RW settles down at $l$ for
at least one time slot.
If the experiment is not successful ($=0$), this jump is not accepted and
RW stays on at $i$ for at least one further time slot; and so on.


\begin{ex} \label{ex:RS-reflection-deterministic}
If for $B \subsetneq \Fset$
\begin{equation*} 
\alpha_j =
\begin{cases}
0 & \text{if}~~~j\in B,\\
1 & \text{if}~~~j\in \Fset \setminus B\,,
\end{cases}
\end{equation*}
then we have exactly the classical reflection at taboo set $B$ as described above, because a jump to
$j\in B$ is never accepted, but whenever a jump to $j\in \Fset \setminus B$ is proposed, this will be accepted with probability $1$.

If we have some general $\avvect$ then the set $B(\avvect)=\{j\in \Fset : \alpha_j =0\}$
is a taboo set for the ''randomized reflection process''.
\end{ex}

\subsubsection{Transition matrix and stationary distribution}\label{sect:RS-reflectionmatrix}
It is easy to see that this construction of a modified chain by randomized reflection
generates a new homogeneous Markov chain.

For the Markov chain $X=(X_n:n\in \N_0)$ with transition matrix $\groutmx=(\groutmx(i,j):i,j\in \Fset)$
and a non zero vector  $\avvect=(\alpha_j:j\in \Fset)$ of acceptance probabilities
denote by  $X^{(\avvect)}$ the Markov chain modification of $X$ under randomized reflection and by $\groutmx^{(\avvect)}$ the transition matrix of $X^{(\avvect)}$, and
by $B(\avvect)=\{j\in \Fset: \alpha_j =0\}\subsetneq \Fset$ the taboo set for  $X^{(\avvect)}$.\\
Then we obtain
\begin{equation}\label{eq:RS-reflection-alpha1}
\groutmx^{(\avvect)}(i,j)=
\begin{cases} \groutmx(i,j)\alpha_j,&
 i,j\in  \Fset,    \:  i\not=j,\\
\groutmx(i,i) + \sum_{k\in \Fset} \groutmx(i,k)(1-\alpha_j),&  i\in \Fset,  \: i=j\,.
\end{cases}
\end{equation}
The Markov chain $X^{(\avvect)}$ with state space $\Fset$ and transition matrix $\groutmx^{(\avvect)}$ may be
 reducible even on $\Fset\setminus B(\avvect)$.\\

A standard  assumption in the literature for applying this reflection principle as rerouting scheme  is that the original routing Markov chain $X$, resp. its transition matrix $\groutmx$ is reversible
for some probability $\eta=(\eta_i: i \in \Fset)$.
We shall set this assumption always in force when investigating this protocol.
To determine the appropriate modified transition matrix $\groutmx^{(\avvect)}$ and the
associated stationary distribution $\eta^{(\avvect)}$ is direct in this case.
\begin{prop}\label{prop:RS-reflection-eta-alpha}
If $\eta$ is the unique steady state distribution of the reversible irreducible Markov chain $X$ on the finite state space $\Fset$, then the Markov chain $X^{(\avvect)}$ is reversible as well with
 steady state distribution $\eta^{(\avvect)}$ with
\begin{equation}
\eta^{(\avvect)}=\left(C^{(\avvect)}\right)^{-1}(\eta_{j}\alpha_{j}:j\in \Fset)
\label{eq:RS-eta-alpha-from-eta-refl}
\end{equation}
with normalization constant
\begin{equation}
C^{(\avvect)}=(\eta I_{\avvect}\mathbf{e}) = \langle\eta,\avvect\rangle\label{eq:RS-eta-alpha-normalization-constant-refl}
\end{equation}
$\eta^{(\avvect)}$ has support $\Fset\setminus B(\avvect)$.
\end{prop}
\begin{proof}
The proof is by directly checking the detailed balance equations for $\groutmx^{(\avvect)}$.
\begin{eqnarray*}
&&(\eta_{j}\alpha_{j})\groutmx^{(\avvect)}(j,i) = (\eta_{i}\alpha_{j})\groutmx^{(\avvect)}(i,j), \qquad i,j\in \Fset\\
&\Leftrightarrow &(\eta_{j}\alpha_{j})\groutmx(j,i)\alpha_{i} = (\eta_{i}\alpha_{i})\groutmx(i,j)\alpha_{j}, \qquad i,j\in \Fset\,.
\end{eqnarray*}
\end{proof}

\subsection{Discussion of randomization algorithms}\label{sect:RS-connection-simulation}
A closer look on our randomization procedures, considered as algorithms to manipulate a random walk and its
stationary distribution, reveals close connections to standard simulation procedures. This will be
discussed in this section and we will compare  furthermore the randomization procedures.

\subsubsection{General sampling schemes}\label{sect:RS-general-sampling}
Our starting point is an ergodic random walk $X$ on $\Fset$ with stationary distribution $\eta$.
$X$ is modified according to an acceptance regime $\avvect$. Our aim is to study the impact of this modification on the stationary distribution.
This is different to the standard problem of simulation, i.e. to generate samples from a given distribution, which is known in principle, but not easy to access. We mention two simulation algorithms which are
structurally related to our procedures.

{\bf (I)} Markov Chain Monte Carlo algorithms  start from a target distribution $\eta$ which is usually not directly accessible for which e.g.
\begin{equation*}
    \int_{\Fset} f(x) \eta (dx) = \sum_{x\in \Fset} f(x) \eta_x\,,
\end{equation*}
has to be computed. The idea is to construct a homogeneous Markov chain with stationary distribution
$\eta$ and to sample from the Markov chains' state distribution after a sufficiently long time horizon.

The construction of the one-step transition kernel following  Hastings or Metropolis
(see \cite{bremaud:99}[Section 7.7.1]) yields a two-step scheme:
\begin{enumerate}
\item from the present state generate a proposal for the next state and decide about acceptance of the proposed state,
\item
if the proposed state is rejected stay on at the present state and restart after a time slot.
\end{enumerate}
This is a procedure as descibed in Proposition \ref{prop:RS-reflection-eta-alpha}. The conclusion is that the rerouting scheme {\bf rs-rd} used to model a transmission network's reaction on full buffers and its
generalization, our randomized reflection, can be considered as MCMC processes.

{\bf (II)}  Von Neumann's acceptance-rejection method for sampling from a complicated distribution $\eta$
(see \cite{bremaud:99}[p. 292]) can be connected to our modification of a Markov chain by randomized skipping. If the random walk of Section \ref{sect:RS-randomskip} is an i.i.d. sequence the randomized skipping procedure is exactly sampling from $(\eta_j \alpha_j:j\in\Fset)$ by sampling from $\eta$ with possible rejection.\\

We remark that in MCMC algorithms usually reversible chains are constructed with acceptance
probabilities which
may depend on the departure state and the proposal. Such generalization is easily constructed here as well,
in our notation we would incorporate success probabilities $\alpha_{ij}$.
Proposition \ref{prop:RS-reflection-eta-alpha} can be modified directly to incorporate this feature.

A similar property and the proof thereof for randomized skipping for non-reversible Markov chains
seems to be not possible without loosing the simple to evaluate steady state distribution.

\subsubsection{Importance sampling}\label{sect:RS-importanceS}
Our randomization procedures resemble obviously importance sampling procedures in simulations,
because we
\begin{enumerate}
  \item produce a weighted version of the probability distribution $\eta$ on $\Fset$, and
  \item define a modified random walk that generates this weighted distribution as its limiting distribution,
\end{enumerate}
and this can therefore be considered as construction of a change of measure.
In importance sampling such change of measure is  used to make the ''more important states''
more often visited by the random walker.
An important problem is to compute expectations of the form
\begin{equation}\label{eq:RS-integration1}
    \int_{\Fset} f(x) \eta (dx) = \sum_{x\in \Fset} f(x) \eta_x\,,
\end{equation}
where  $X=(X_n:n\in \N_0)$ with transition matrix $r$ is ergodic with limiting
distribution $\eta$, where $\eta$ is not known or not easily accessible.

For simplicity we assume $\alpha_j>0$ for all $j\in \Fset$ and use the notation of the
previous sections. Recall that $C^{(\avvect)}=\langle\eta,\avvect\rangle$ is the normalization
constant of $\eta^{(\avvect)}.$
By simple formal manipulation
\begin{equation*}
 \sum_{x\in \Fset} f(x) \eta_x ~=~   \sum_{x\in \Fset} \frac{f(x)}{\langle\eta,\avvect\rangle^{-1}}
  \cdot \frac{\eta_x}{\langle\eta,\avvect\rangle}
   ~=~ \frac{ \sum_{x\in \Fset}\frac{f(x)}{\alpha_x}\cdot \frac{\eta_x\alpha_x}{\langle\eta,\avvect\rangle}}
   {\sum_{x\in \Fset}\frac{1}{\alpha_x}\cdot\frac{\eta_x\alpha_x}{\langle\eta,\avvect\rangle}}\,.
\end{equation*}

Nominator and denominator are integrals with respect to the  stationary distribution of $X^{(\avvect)}$ derived from $X$ with acceptance probability $\avvect = (\alpha_x: x \in \Fset)$.
$X$ is ergodic and $\alpha_x>0$ $\forall x \in \Fset$ implies
that $X^{(\avvect)}$ is ergodic on $\Fset$ as well.

Denote by $T(i)$ the first entrance time into $i$, then from the regenerative structure of $X^{(\avvect)}$
we have

\begin{equation*}
 \sum_{x\in \Fset}\frac{f(x)}{\alpha_x}\cdot \frac{\eta_x\alpha_x}{\langle\eta,\avvect\rangle}
 = \frac{
 E^{(\avvect)}\left[\left.
 \sum_{n=0}^{T(i)-1}\frac{f(X^{(\avvect)}_n)}{\alpha_{X^{(\avvect)}_n}}
 \right|
 X^{(\avvect)}_0=i
 \right]
 }{
 E^{(\avvect)}\left[\left.
 T(i)
 \right|
 X^{(\avvect)}_0=i
 \right]
 }
\end{equation*}
and
\begin{equation*}
\sum_{x\in \Fset}\frac{1}{\alpha_x}\cdot\frac{\eta_x\alpha_x}{\langle\eta,\avvect\rangle}
=
\frac{
 E^{(\avvect)}\left[\left.
 \sum_{n=0}^{T(i)-1}\frac{1}{\alpha_{X^{(\avvect)}_n}}
 \right|
 X^{(\avvect)}_0=i
 \right]
 }{
 E^{(\avvect)}\left[\left.
 T(i)
 \right|
 X^{(\avvect)}_0=i
 \right]
 }\,.
\end{equation*}
This yields eventually
\begin{equation*}
\sum_{x\in \Fset}{f(x)}{\eta_x}
=
\frac{
 E^{(\avvect)}\left[\left.
 \sum_{n=0}^{T(i)-1}\frac{f(X^{(\avvect)}_n)}{\alpha_{X^{(\avvect)}_n}}
 \right|
 X^{(\avvect)}_0=i
 \right]
 }{
 E^{(\avvect)}\left[\left.
 \sum_{n=0}^{T(i)-1}\frac{1}{\alpha_{X^{(\avvect)}_n}}
 \right|
 X^{(\avvect)}_0=i
 \right]
 }\,.
\end{equation*}

To obtain an estimator for $E^{(\avvect)}\left[\left.
 \sum_{n=0}^{T(i)-1}g\left(X^{(\avvect)}_n\right)
 \right|
 X^{(\avvect)}_0=i
 \right]$
denote $T^{0}(i)\equiv 0$, and
$T^{(n+1)}(i):=\inf(k>T^{(n)}(i): X^{(\avvect)}_k=i)$,
$n\geq 0$,
and take the sample means of the independent replications
\begin{equation*}
\frac{1}{K}\sum_{k=1}^{K}\left(\left.
 \sum_{n=T^{(k-1)}(i)}^{T^{(k)}(i)}
 g\left(X^{(\avvect)}_n\right)
 \right|
 X^{(\avvect)}_0=i
 \right)\,,
\end{equation*}
which converge a.s. to the target expression for
$K\rightarrow \infty$.

So, if it is possible to simulate without too much effort the system which is described by the
Markov chain $ X^{(\avvect)}$, we can approximate the integral by the quotient of the two time averages over the regeneration periods.

We remark that by our simple construction we not only can make important states more probable to visit, but for a given initial
state $i$ we can control the expected length of the regeneration period.\\

If $X$ is reversible, the one step transition matrix of the chain modified  by randomized
reflection is directly accessible, see \eqref{eq:RS-reflection-alpha1}.

In the general case when $X$ is modified by randomized skipping the one step transition matrix
of the  modified chain is possibly not directly available, see \eqref{eq:RS-r-alpha2}.
 We can remedy this drawback by simulating instead the auxiliary chain $(X^{(A)},Y^{(A)})$
(with obvious modifications) which is used to find the one step transition matrix for
randomized skipping in Section \ref{sect:RS-skippingmatrix}.
The  modification of $(X^{(A)},Y^{(A)})$ is not to make $\Fset\times \{1\}$ absorbing.
The transition rates \eqref{eq:RS-Amatrix3} and \eqref{eq:RS-Amatrix4} are replaced by
\begin{eqnarray}
P(X^{(A)}_{n+1}=j,Y^{(A)}_{n+1}=1|X^{(A)}_{n}=i,Y^{(A)}_{n}=1) & = &
\groutmx(i,j) \alpha_{j} ,\label{eq:RS-Amatrix7}\\
P(X^{(A)}_{n+1}=j,Y^{(A)}_{n+1}=0|X^{(A)}_{n}=i,Y^{(A)}_{n}=1) & = &
 \groutmx(i,j) (1-\alpha_{j})\,.\label{eq:RS-Amatrix8}
\end{eqnarray}

In the time average over a regeneration period we then have to sum instead of
 ${f(X^{(\avvect)}_n(\omega))}$ the values ${f(X^{(A)}_n(\omega))\cdot Y^{(A)}_n(\omega)}$
and instead of $\frac{1}{\alpha_{X^{(\avvect)}_n(\omega)}}$ the values
$\frac{Y^{(A)}_n(\omega)}{\alpha_{X^{(\avvect)}_n(\omega)}}$
and to divide
by the expected sojourn time in $\Fset\times\{1\}$  until return to $(i,1)$ given start in  $(i,1)$
(which cancels out).

\subsubsection{Comparison of randomized skipping and randomized reflection}\label{sect:RS-comparison-sk-refl}
We can compare transition matrices having the same dimension and stationary distribution
by Peskun ordering (see \cite{peskun:73}) which is standard ordering in MCMC simulations.
\begin{defn}\label{def:RS-peskun}
Let $\groutmx=(  \groutmx{(i,j)}:i,j\in \Fset)$ and $ \groutmx'=(\groutmx'{(i,j)}:i,j\in \Fset)$ be transition
matrices on a finite set $\Fset$ such that $\xi \groutmx=\xi  \groutmx'=\xi$ for a probability vector $\xi$.

We say that $ \groutmx'$  is smaller than $\groutmx$ in the Peskun order,  $ \groutmx' \prec_P \groutmx$,
if for all $j, i\in \Fset$ with $i\neq j$  holds $ \groutmx'{(j,i)}\leq  \groutmx{(j,i)}$.
\end{defn}

Peskun used this order to compare reversible transition matrices with the
same stationary distribution  and their asymptotic variance. A useful interpretation is that
in case of $ \groutmx' \prec_P \groutmx$ a random walker under $\groutmx$ explores his state space faster than
a random walker under $\groutmx'$.

\begin{prop}\label{prop:RS-peskun-rsk-rrefl}
Let $X$ be irreducible with transition matrix $\groutmx$ on the finite state space $\Fset$ which is reversible for the stationary distribution  $\eta=(\eta_j:j\in \Fset)$, and  $\avvect=(\alpha_j:j\in \Fset)$ be
a non zero vector.\\
Denote by $\groutmx^{(\avvect)}_{skip}$ the modification of $\groutmx$ by randomized skipping
and by $\groutmx^{(\avvect)}_{refl}$ the modification of $\groutmx$ by randomized reflection.
Then it holds $\groutmx^{(\avvect)}_{refl} \prec_P \groutmx^{(\avvect)}_{skip}$.
\end{prop}
\begin{proof}
The proof is obvious because for all $j, i\in \Fset$ with $i\neq j$  holds
$\groutmx^{(\avvect)}_{refl}(i,j) = \groutmx(i,j)\cdot \alpha_j$ and
$\groutmx^{(\avvect)}_{skip}(i,j) = \groutmx(i,j)\cdot \alpha_j + \text{possible further terms, which are all non-negative}$.
\end{proof}

Consequences of  $\groutmx^{(\avvect)}_{refl} \prec_P \groutmx^{(\avvect)}_{skip}$ in Proposition \ref{prop:RS-peskun-rsk-rrefl} are
\begin{enumerate}
\item the asymptotic variance in the central limit theorem for a Markov chain driven by
$\groutmx^{(\avvect)}_{skip}$ is smaller than for a Markov chain driven by  $\groutmx^{(\avvect)}_{refl}$
although both chains have the same limiting distribution (see \cite{tierney:98}).
\item
the spectral gap of a Markov chain driven by
$\groutmx^{(\avvect)}_{refl}$ is smaller than for a Markov chain driven by  $\groutmx^{(\avvect)}_{skip}$
although both chains have the same limiting distribution.
This means that the modification of $\X$ by randomized skipping converges faster to equilibrium
than the modification of $\X$ by randomized reflection.
\end{enumerate}

A further distinction between modifications of $\X$ by randomized skipping and by randomized reflection
is visible if we consider for a reversible Markov chain $\X$ with transition matrix $\groutmx$
the associated transition graph $G=(\Fset, {\cal E})$ where
$(i,j)\in {\cal E} \Leftrightarrow \groutmx(i,j)>0$.

Note that from reversibility follows $(i,j)\in {\cal E} \Leftrightarrow (j,i)\in {\cal E}$.
Denote by ${\cal N}(i) := \{j\in \Fset: (i,j)\in {\cal E}\}$ the one-step neighbourhood of $i$.

If we consider randomized skipping and by randomized reflection as algorithms to produce
modifications of the random walk $\X$,
then the algorithm for randomized reflection is local with respect to $G$ because  transitions out of
states $i$ are determined only on the basis of decisions  in ${\cal N}(i)$.

On the other side, randomized skipping is global with respect to $G$ because for a transition out of $i$
we possibly must consider random experiments anywhere on the graph.\\

To construct local algorithms for control, scheduling, and development of complex systems is important for key applications, e.g. wireless sensor networks or autonomous interacting robotic systems.

For mathematical investigations of this problem we refer to \cite{shah;shin:12}, where an
optimal local control algorithm for  wireless sensor networks is determined by local approximation
of an optimal global algorithm.

\section{Stochastic networks in a random environment}
\label{RS-sect:networks}

\subsection{Jackson networks}\label{sect:RS-jackson-networks}
We consider a Jackson network \cite{jackson:57} with node set
$\Jset:= \{1,\dots,J\}$. Customers arrive in independent external Poisson streams,
 at node $j$  with finite intensity $\lambda_j \geq 0,$ we set
$\lambda =\lambda_1+\ldots +\lambda_J> 0$. Customers are indistinguishable
and follow the same rules.
Customers request for service which is exponentially distributed with mean
$1$ at all nodes, all these requests constitute an independent family
of variables which are  independent of the arrival streams.

Nodes are exponential single servers with state
dependent service rates and with an infinite waiting room under
first--come--first--served (FCFS)
regime. If at node $j$ there are $n_j>0$ customers present, either
in service or waiting, then service is provided there with intensity
$\mu_j(n_j)>0$.
(Therefore, in general, the obtained service time is not
exponential-$1$.)

Routing is Markovian,
a customer departing from node $i$ immediately proceeds to node
$j$ with probability $\routmx{(i,j)}\geq 0$, and departs from the network
with probability $\routmx(j,0).$ Taking
$\routmx(0,j)=\lambda_j/\lambda,\
\routmx(0,0)=0$, we assume that the (extended) routing matrix
$\routmx=(\routmx(i,j):{i,j=0,\ldots,J})$ is irreducible. This ensures that
the traffic equations
\begin{equation} \label{eq:RS-traffic}
\eta_j=\lambda_j + \sum_{i=1}^J \eta_i \routmx(i,j),\quad\quad
j=1,\dots,J,
\end{equation}
have a unique solution which we denote by $\eta =
(\eta_j:j=1,\dots,J).$
We extend the traffic equation \eqref{eq:RS-traffic} to a steady state equation for
a routing Markov chain by
\begin{equation} \label{eq:RS-traffic-routing}
\eta_j=\sum_{i=0}^J \eta_i \routmx(i,j),\quad\quad
j=0,1,\dots,J,
\end{equation}
which is solved by $\eta = (\eta_j:j=0,1,\dots,J),$
where $\eta_0 := \lambda$ and the other $\eta_j$ are from \eqref{eq:RS-traffic}.\\
We use $\eta$ in both meanings and emphasize the latter one
by {\em extended traffic solution} $\eta$. $\eta$ is usually not a stochastic vector.
Let $\X=(X(t):t\geq 0)$ denote the vector process recording the
queue lengths in the network for time  $t$.
$X(t)=(X_1(t),\dots,X_J(t))\in \N^{\Jset}_0$ reads: at time $t$ there are
$X_j(t)$ customers present at node $j$, either in service or
waiting. The assumptions put on the system imply that $\X$ is a
strong Markov process on state space $\N^{\Jset}_0$  with generator
$Q^\X = (q^\X(\nvect,\nvect^{'}):\nvect,\nvect^{'}\in\mathbb{N}_{0}^{\Jset})$.
The strict positive transition rates of $Q^\X$ which describe the physical actions of the network
 are for $\mathbf{n}=(n_1,\dots,n_J)\in \mathbb{N}_{0}^{\Jset}$
\begin{align}\label{eq:RS-GB-JN}
q^\X(\nvect\qsep\nvect+\evect_{i})
& =  \lambda \routmx(0,i),
& i\in\Jset\,,\\
q^\X(\nvect\qsep\nvect-\evect_{j}+\evect_{i})
& = 1_{[n_{j}>0]}\mu_{j}(n_{j})\routmx(j,i),
& i\neq j,\, i, j\in\Jset \,, \nonumber\\
q^\X(\nvect\qsep\nvect-\evect_{j})
& = 1_{[n_{j}>0]}\mu_{j}(n_{j})\routmx(j,0),
& j\in\Jset\nonumber
\,.
\end{align}

For an ergodic network process $\X$ Jackson's theorem \cite{jackson:57} states that
the unique steady state and
limiting distribution $\xi$ on $ \N^{\Jset}_0$ is
\begin{equation}                        \label{eq:RS-jackson-steadystate}
\xi(\nvect)=\xi(n_1,\dots,n_J) = \prod_{j=1}^{J} \prod_{k=1}^{n_j} \frac{\eta_j}{\mu_j(k)} C(j)^{-1}
\end{equation}
with normalizing  constants $C(j)$ for the  marginal (over nodes) distributions of $\X$.\\

\subsection{Modified Jackson networks}
\label{sect:RS-modified-jackson-networks}
We consider in this section the problem to adjust  routing to
changes of  service capacities of the network and
assume that the routing regime is constructed to meet some optimization criterion with respect to the local queue lengths $X_j$ for the fixed $\mu_j(n_j),$ $j\in \Jset$ . This means that the
ratios $\eta_j/\mu_j(n_j)$ reflect a locally optimal behaviour of the nodes and should be
maintained when service capacities change.

\subsubsection{Degraded service and adapted routing}\label{sect:RS-degraded}

If the service intensities $\mu_{i}(n_{i})$  at node $i$ are degraded by a factor
$\srfactor_i\in [0,1]$ for $i\in \Jset$, we have to readjust the routing regime in a way to
meet these ratios again, at least approximately.
It is possible that some nodes can break down
completely, i.e. $\srfactor_\ell=0$ for such node $\ell$.
Clearly, the broken down nodes should not be visited any longer, and from the side constraint to
maintain  at least approximately the ratios ``overall arrival rate''/``service rates'', it is tempting
to try a rerouting by randomized skipping or  reflection to some suitably selected
``acceptance probability vector'' according to the $\avvect=\avvect(\srfactorvect)$ (If there is no ambiguity we will shortly write only $\avvect$.)\\
Then the new routing will have two components:

\begin{enumerate}
  \item Part of the total external arrival rate will be rejected, and
  \item the admitted load will be redistributed among the nodes which are not completely broken
  down in a way to meet exactly the old ratios.
\end{enumerate}

We define  $\alpha_0=1$ and take the given $\gamma_i$, $i\in \Jset,$ to define the vector
$\avvect(\srfactorvect)=(\alpha_i, i\in \Jset_0)$ of acceptance probabilities
as $\alpha_i := \gamma_i$ and modify the extended routing
matrix $\routmx=(\routmx(i,j):i,j\in \Jset_0)$ according to the randomized skipping, respectively
randomized reflection  principle from Sections \ref{sect:RS-skippingmatrix} and \ref{sect:RS-reflectionmatrix}.

It was surprising to us that we have to choose $\alpha_j=\srfactor_j$ to adjust the acceptance probability of  node $j$ precisely to the service rate factor
$\srfactor_j \in [0,1]$. This will no longer be possible if we later on will allow more general factors
$\srfactor_j \in [0,\infty)$, which may lead to acceptance probabilities $\alpha_j\neq \srfactor_j$.

The modified routing matrix is in both cases  (randomized skipping, resp.  reflection) denoted by
$\routmx^{(\avvect)}=(\routmx^{(\avvect)}(i,j):i,j\in \Jset_0)$, and is given alternatively
according to
\prettyref{thm:RS-r-alpha}, respectively \prettyref{prop:GS-y-from-eta}.
Note, that from $(\alpha_0=1)$ we obtain $\eta^{(\avvect)}(0) = \eta(0)$.

\begin{defn}\label{def:RS-JBA}
The set of ''blocked nodes''   $\JsetB{\srfactorvect}\subset \Jset$
is defined by
\[
j\in \JsetB{\srfactorvect}\Longleftrightarrow\srfactor_{j}=0\,,
\]
and its complement in $\Jset$, the set of the ''working nodes'' $\JsetW{\srfactorvect}\subset \Jset$, is defined by
\[
j\in \JsetW{\srfactorvect}\Longleftrightarrow\srfactor_{j}>0\,.
\]
\end{defn}

When the service rates of a Jackson network are modified according to $\srfactorvect$
and the routing is adjusted according to $\avvect(\srfactorvect)$
we obtain a new network process
$\X^{(\srfactor)}=(X^{(\srfactorvect)}(t):t\geq 0)$ as the vector process recording the
queue lengths in the network.
$X^{(\srfactorvect)}(t)=(X^{(\srfactorvect)}_1(t),\dots,X^{(\srfactorvect)}_J(t))\in \N^{\Jset}_0$ reads: at time $t$ there are
$X^{(\srfactorvect)}_j(t)$ customers present at node $j$, either in service or
waiting. The assumptions put on the system imply that $\X$ is a
strong Markov process on state space $\N^{\Jset}_0$  with generator
$Q^{\X^{(\srfactorvect)}} = : Q^{{(\srfactorvect)}} = (q^{(\srfactorvect)}(\mathbf{n},\mathbf{n}^{'}):\mathbf{n},\mathbf{n}^{'}\in\mathbb{N}_{0}^{\bar{J}})$.
The strict positive transition rates of $Q^{(\srfactorvect)}$ are
 under both rerouting regimes for $\mathbf{n}=(n_1,\dots,n_J)\in \mathbb{N}_{0}^{\bar{J}}$
\begin{align}\label{eq:RS-GB-JN-alpha}
q^{(\srfactorvect)}(\nvect\qsep\nvect+\evect_{i})
& = \lambda \routmx^{(\avvect)}(0,i),
& i\in\Jset_0\, ,\\
q^{(\srfactorvect)}(\nvect\qsep\nvect-\evect_{j}+\evect_{i})
& = 1_{[n_{j}>0]}\srfactor_j  \mu_{j}(n_{j})\routmx^{(\avvect)}(j,i)
& i,j \in\Jset,\, i\neq j \,,
\nonumber\\
q^{(\srfactorvect)}(\nvect\qsep\nvect-\evect_{j})
& = 1_{[n_{j}>0]}\srfactor_j \mu_{j}(n_{j})\routmx^{(\avvect)}(j,0),
& j\in\Jset\nonumber
\,.
\end{align}

That the construction is successful  in maintaining the ratios (overall arrival rate/service rates) shows the next theorems, which will be proved almost simultaneously. Recall $0/0:=0$.


\begin{theorem}
\label{thm:RS-JN-GB-skip-degrading}
{\sc[Modified Jackson networks: Randomized skipping and degrading nodes]}
Let $\X$ be an  ergodic  Jackson network process as described in Section \ref{sect:RS-jackson-networks}
with stationary distribution $\xi$ from \eqref{eq:RS-jackson-steadystate}.
If the service intensities $\mu_{i}(n_{i})$  at node $i$ are degraded by a factor
$\srfactor_i\in [0,1]$ for $i\in \Jset$, and routing is changed to $\routmx^{(\avvect)},$ with  $\avvect=\avvect(\srfactorvect)=(\alpha_i:i\in \Jset_0)$
where $\alpha_0=1$
and $\alpha_j = \srfactor_j$ for $j\in \Jset$,  by randomized skipping according
to Theorem \ref{thm:RS-r-alpha},  then
 $\xi$ is a stationary distribution for $\X^{(\srfactorvect)}=(X^{(\srfactorvect)}(t):t\geq 0)$ as well.\\
Moreover, if  $\JsetB{\srfactorvect} = \emptyset$, then $\X^{(\srfactorvect)}$ is ergodic.\\
If  $\JsetB{\srfactorvect}\neq \emptyset$, then $\X^{(\srfactorvect)}$ is not irreducible on
$\mathbb{N}_{0}^{\bar{J}}$ and its state space is divided into an infinite set of closed subspaces
\begin{equation*}
\mathbb{N}_{0}^{\JsetW{\srfactorvect}}\times \{(n_j:j\in\JsetB{\srfactorvect})\}\quad\forall
(n_j:j\in\JsetB{\srfactorvect})\in \mathbb{N}_{0}^{\JsetB{\srfactorvect}}\,.
\end{equation*}
For any probability distribution $\varphi$ on
$\mathbb{N}_{0}^{\JsetB{\srfactorvect}}$
there exists a stationary distribution $\xi^{(\srfactorvect)}_\varphi$ for $\X^{(\srfactorvect)}$, which is  for
$\nvect=(n_1,\dots,n_J)\in \mathbb{N}_{0}^{\Jset}$
\begin{equation*} 
\xi^{(\srfactorvect)}_\varphi(\nvect)=\xi^{(\srfactorvect)}_\varphi(n_1,\dots,n_J) =
\prod_{j\in \JsetW{\srfactorvect}} \prod_{k=1}^{n_j} \frac{\eta_j}{\mu_j(k)} C(j)^{-1}\cdot
\varphi(n_j:j\in\JsetB{\srfactorvect})\,.
\end{equation*}
\end{theorem}

\begin{ex}\label{ex:RS-jackson-skipping-degrading}
Consider a Jackson network on a node set $\Jset=\{1,2,3,4\}$ with a routing matrix $\routmx$ as in  \prettyref{ex:RS-routing-skipping} with service rate  factors $\srfactor=(1,0.5,1,1)$. Then the availability vector is $\avvect=(1,1,0.5,1,1)$
and a new modified routing matrix  $\routmx^{(\avvect)}$ is
as in \prettyref{ex:RS-routing-skipping}.

See \prettyref{fig:RS-jackson-skipping-degrading}.
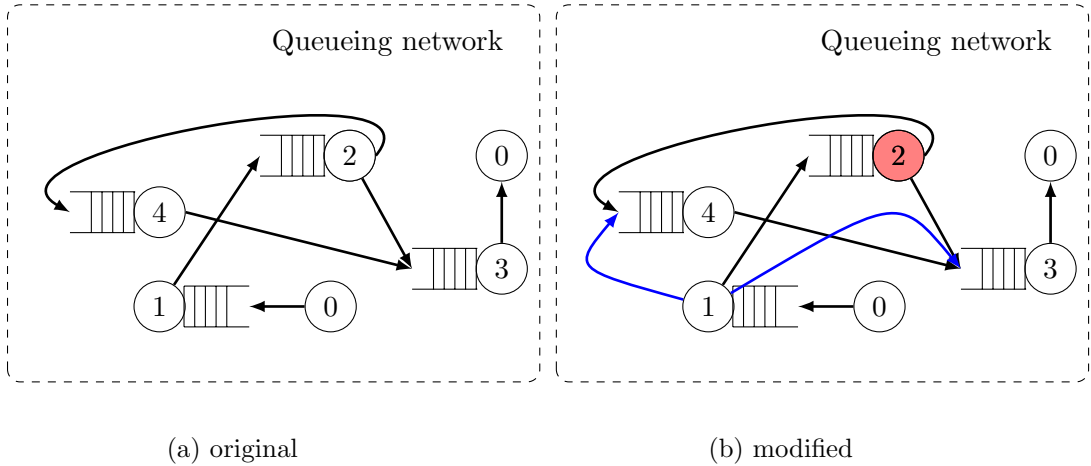
\begin{figure}[H]
\centering
\begin{subfigure}[b]{0.45\textwidth}
\begin{tikzpicture}

\newcommand{\yslant}{0.2} 
\newcommand{\xslant}{-0.5}
\newcommand{\height}{6}

\begin{scope}


\end{scope}

\begin{scope}
  \node()[] at (3.0,4) {Queueing network};
  \draw[dashed,rounded corners=1ex] (-2,-0.5) rectangle (5,4.5);

  \node(0-1-orgn)[] at (2.25,0.5) {};
  \node(0-2-orgn)[] at (4.5,2.5) {};
  \node(1-orgn)[] at (0,0.5) {};
  
  \node(1-Q-orgn)[] at (0.75,0.5) {};
  \node(2-orgn)[] at (2.5,2.5) {};
  \node(2-Q-orgn)[] at (1.75,2.5) {};
  \node(2-orgn-under)[] at (2.5,2) {};
  \node(2-orgn-1)[] at (3,2.5) {}; 
  \node(2-orgn-2)[] at (3,3) {};
  \node(2-orgn-3)[] at (0,3) {};
  \node(3-orgn)[] at (4.5,1) {};
  \node(3-orgn-over)[] at (4.5,4) {};
  \node(3-Q-orgn)[] at (3.75,1) {};
  \node(4-orgn)[] at (0,1.75) {};
  \node(4-Q-orgn)[] at (-0.75,1.75) {};
  \node(4-orgn-over)[] at (0,4) {};
  \node(4-orgn-under)[] at (-1.8,1) {};
  \node(5-orgn)[] at (0,1.75) {};

  \node(0-1)[draw, shape=circle, color=external.fg.color] at (0-1-orgn) {$0$};
  \node(0-2)[draw, shape=circle, color=external.fg.color] at (0-2-orgn) {$0$};
  \node(1)[draw, shape=circle] at (1-orgn) {$1$};
  \node(2-white)[draw, shape=circle, fill=white] at (2-orgn) {$2$};
  \node(2)[draw, shape=circle, fill=white] at (2-orgn) {$2$};

  \node(3)[draw, shape=circle] at (3-orgn) {$3$};
  \node(4)[draw, shape=circle] at (4-orgn) {$4$};

\node(1-Q)[draw, shape=queue, queue head=west, queue size=infinite,
minimum width =24 pt, minimum height=15pt] at(1-Q-orgn) {};
\node(2-Q)[draw, shape=queue, queue head=east, queue size=infinite,
minimum width =24 pt, minimum height=15pt] at(2-Q-orgn) {};
\node(3-Q)[draw, shape=queue, queue head=east, queue size=infinite,
minimum width =24 pt, minimum height=15pt] at(3-Q-orgn) {};
\node(4-Q)[draw, shape=queue, queue head=east, queue size=infinite,
minimum width =24 pt, minimum height=15pt] at(4-Q-orgn) {};

\draw[arrows={-latex}, thick,
  line width=\customerslinewidth] (1) -- (2-Q.west);
\draw[arrows={-latex}, thick,
  line width=\customerslinewidth] (2) -- (3-Q.west);
\draw[arrows={-latex}, thick,
  line width=\customerslinewidth] (3) -- (0-2);
\draw[arrows={-latex}, thick,
  line width=\customerslinewidth] (0-1) -- (1-Q.east);
\draw[arrows={-latex}, thick,
  line width=\customerslinewidth] (2.east) to [bend right=140] (4-Q.west);
\draw[arrows={-latex}, thick,
  line width=\customerslinewidth] (4.east) -- (3-Q.west);

\end{scope}

\end{tikzpicture}
\caption{original}
\end{subfigure}
\begin{subfigure}[b]{0.45\textwidth}
\begin{tikzpicture}

\newcommand{\yslant}{0.2} 
\newcommand{\xslant}{-0.5}
\newcommand{\height}{6}

\begin{scope}


\end{scope}

\begin{scope}
  \node()[] at (3.0,4) {Queueing network};
  \draw[dashed,rounded corners=1ex] (-2,-0.5) rectangle (5,4.5);

  \node(0-1-orgn)[] at (2.25,0.5) {};
  \node(0-2-orgn)[] at (4.5,2.5) {};
  \node(1-orgn)[] at (0,0.5) {};
  
  \node(1-Q-orgn)[] at (0.75,0.5) {};
  \node(2-orgn)[] at (2.5,2.5) {};
  \node(2-Q-orgn)[] at (1.75,2.5) {};
  \node(2-orgn-under)[] at (2.5,2) {};
  \node(2-orgn-1)[] at (3,2.5) {}; 
  \node(2-orgn-2)[] at (3,3) {};
  \node(2-orgn-3)[] at (0,3) {};
  \node(3-orgn)[] at (4.5,1) {};
  \node(3-Q-orgn)[] at (3.75,1) {};
  \node(4-orgn)[] at (0,1.75) {};
  \node(4-Q-orgn)[] at (-0.75,1.75) {};
  \node(4-orgn-under)[] at (-1.8,1) {};
  \node(5-orgn)[] at (0,1.75) {};

  \node(0-1)[draw, shape=circle, color=external.fg.color] at (0-1-orgn) {$0$};
  \node(0-2)[draw, shape=circle, color=external.fg.color] at (0-2-orgn) {$0$};
  \node(1)[draw, shape=circle] at (1-orgn) {$1$};
  \node(2)[draw, shape=circle, fill=down.partially.bg.color] at (2-orgn) {$2$};
  \node(2)[draw, shape=circle] at (2-orgn) {$2$};
  \node(3)[draw, shape=circle] at (3-orgn) {$3$};
  \node(4)[draw, shape=circle] at (4-orgn) {$4$};

\node(1-Q)[draw, shape=queue, queue head=west, queue size=infinite,
minimum width =24 pt, minimum height=15pt] at(1-Q-orgn) {};
\node(2-Q)[draw, shape=queue, queue head=east, queue size=infinite,
minimum width =24 pt, minimum height=15pt] at(2-Q-orgn) {};
\node(3-Q)[draw, shape=queue, queue head=east, queue size=infinite,
minimum width =24 pt, minimum height=15pt] at(3-Q-orgn) {};
\node(4-Q)[draw, shape=queue, queue head=east, queue size=infinite,
minimum width =24 pt, minimum height=15pt] at(4-Q-orgn) {};

\draw[arrows={-latex}, thick,
  line width=\customerslinewidth] (1) -- (2-Q.west);
\draw[arrows={-latex}, thick,
  line width=\customerslinewidth] (2) -- (3-Q.west);
\draw[arrows={-latex}, thick,
  line width=\customerslinewidth] (3) -- (0-2);
\draw[arrows={-latex}, thick,
  line width=\customerslinewidth] (0-1) -- (1-Q.east);
\draw[arrows={-latex}, thick,
  line width=\customerslinewidth] (2.east) to [bend right=140] (4-Q.west);
\draw[arrows={-latex}, thick,
  line width=\customerslinewidth] (4.east) -- (3-Q.west);

\draw[arrows={-latex}, thick, line width=\customerslinewidth,color=blue] (1) .. controls (2-orgn-under) .. (3-Q.west);

\draw[arrows={-latex}, thick, line width=\customerslinewidth,
color=blue,bend right=40] (1) .. controls (4-orgn-under) ..(4-Q.west);

\end{scope}

\end{tikzpicture}
\caption{modified}
\end{subfigure}
\caption{Original and modified  according to the skipping rule Jackson networks from \prettyref{ex:RS-jackson-skipping-degrading}.
\label{fig:RS-jackson-skipping-degrading}
}
\end{figure}
\end{ex}

\begin{theorem}
\label{thm:RS-JN-GB-refl-degrading}
{\sc [Modified Jackson networks: Randomized reflection
and degrading nodes]}
Let $\X$ be an  ergodic  Jackson network process as described in Section \ref{sect:RS-jackson-networks}
with reversible routing matrix $r$ and
with stationary distribution $\xi$ from \eqref{eq:RS-jackson-steadystate}.
If the service intensities $\mu_{i}(n_{i})$  at node $i$ are degraded by a factor
$\srfactor_i\in [0,1]$ for $i\in \Jset$, and routing is changed to $\routmx^{(\avvect)}$, with  $\avvect=\avvect{\srfactorvect}=(\alpha_i:i\in \Jset_0)$ where $\alpha_0=1$ and
$\srfactor_j = \alpha_j$ for $j\in \Jset$,
by randomized reflection according
to Proposition \ref{prop:RS-reflection-eta-alpha},  then
 $\xi$ is a stationary distribution for $\X^{(\srfactorvect)}=(X^{(\srfactorvect)}(t):t\geq 0)$ as well.\\
Moreover, if $\JsetB{\srfactorvect} = \emptyset$, then $\X^{(\srfactorvect)}$ is ergodic.\\
If $\JsetB{\srfactorvect}\neq \emptyset$, then $\X^{(\srfactorvect)}$ is not irreducible on
$\mathbb{N}_{0}^{\Jset}$ and its state space is divided into an infinite set of closed subspaces
\begin{equation*}
\mathbb{N}_{0}^{\JsetW{\srfactorvect}}\times \{(n_j:j\in\JsetB{\srfactorvect})\}\quad\forall
(n_j:j\in\JsetB{\srfactorvect})\in \mathbb{N}_{0}^{\JsetB{\srfactorvect}}\,.
\end{equation*}
For any probability distribution $\varphi$ on
$\mathbb{N}_{0}^{\JsetB{\srfactorvect}}$
there exists a stationary distribution $\xi^{(\srfactorvect)}_\varphi$ for $\X^{(\srfactorvect)}$, which is  for
$\nvect=(n_1,\dots,n_J)\in \mathbb{N}_{0}^{\bar{J}}$
\begin{equation*} 
\xi^{(\srfactorvect)}_\varphi(\mathbf{n})=\xi^{(\srfactorvect)}_\varphi(n_1,\dots,n_J) =
\prod_{j\in \JsetW{\srfactorvect}} \prod_{k=1}^{n_j} \frac{\eta_j}{\mu_j(k)} C(j)^{-1}\cdot
\varphi(n_j:j\in\JsetB{\srfactorvect})\,.
\end{equation*}
\end{theorem}

\begin{proof}[Proof of \prettyref{thm:RS-JN-GB-skip-degrading} and
\prettyref{thm:RS-JN-GB-refl-degrading}]
The global balance equation ${x}\cdot Q^{(\srfactorvect)}=0$ for the joint queue length process
$X^{(\srfactorvect)}$ of the modified system is in both settings for $\mathbf{n}=(n_1,\dots,n_J) \in \N_0^{\Jset}$
\begin{eqnarray}
&  & x(\nvect)\left(\sum_{j\in\Jset}\lambda \routmx^{(\avvect)}(0,j)+\sum_{j\in\Jset}1_{[n_{j}>0]}\srfactor_j\mu_{j}(n_{j})
(1-\routmx^{(\avvect)}(j,j))\right)\nonumber \\
 & = & \sum_{i\in\Jset} x(\nvect-\evect_{i})1_{[n_{i}>0]}\lambda\routmx^{(\avvect)}(0,i)
  +\sum_{j\in\Jset}
  \sum_{i\in\Jset\backslash\{j\}}x(\nvect-\evect_{i}+\evect_{j})1_{[n_{i}>0]}
  \srfactor_j\mu_{j}(n_{j}+1)\routmx^{(\avvect)}(j,i)\nonumber \\
 &  & +\sum_{j\in\Jset}x(\nvect+\evect_{j})\srfactor_j\mu_{j}(n_{j}+1)
 \routmx^{(\avvect)}(j,0)\,.\label{eq:RS-LS-RAND-X-alpha-gbe}
\end{eqnarray}

We first consider  the case $\JsetB{\srfactorvect}\neq \emptyset$. Then for $i\in \JsetB{\srfactorvect}$
we have $\srfactor_i=\alpha_i=0$ and $\routmx^{(\avvect)}(j,i) =0$ for all $j\in\Jset_0$  and \eqref{eq:RS-LS-RAND-X-alpha-gbe} reduces to
\begin{eqnarray}
&  & x(\nvect)\left(\sum_{j\in\JsetW{\srfactorvect}} \lambda \routmx^{(\avvect)}(0,j)
+\sum_{j\in\JsetW{\srfactorvect}}1_{[n_{j}>0]}\srfactor_j\mu_{j}(n_{j})
(1-\routmx^{(\avvect)}(j,j))\right)\nonumber \\
 & = & \sum_{i\in\JsetW{\srfactorvect}} x(\nvect-\evect_{i})1_{[n_{i}>0]}\lambda\routmx^{(\avvect)}(0,i) \nonumber \\
& &  +\sum_{j\in\JsetW{\srfactorvect}}\sum_{i\in\JsetW{\srfactorvect}
 \backslash\{j\}}x(\nvect-\evect_{i}+\evect_{j})1_{[n_{i}>0]}\srfactor_j
 \mu_{j}(n_{j}+1)\routmx^{(\avvect)}(j,i)\nonumber \\
 &  & +\sum_{j\in\JsetW{\srfactorvect}}x(\nvect+\evect_{j})
 \srfactor_j\mu_{j}(n_{j}+1)
 \routmx^{(\avvect)}(j,0)\,.\label{eq:RS-LS-RAND-X-alpha-gbe-reduced1}
\end{eqnarray}
Inserting $x(n_1,\dots,n_J) = \prod_{j\in \JsetW{\srfactorvect}} \prod_{k=1}^{n_j} \frac{\eta_j}{\mu_j(k)} C(j)^{-1}\cdot
\varphi(n_j:j\in\JsetB{\srfactorvect})$ for any probability density $\varphi$ on
$\mathbb{N}_{0}^{\JsetB{\srfactorvect}}$ we see that immediately
$\prod_{j\in \JsetW{\srfactorvect}} C(j)^{-1}\cdot \varphi(n_j:j\in\JsetB{\srfactorvect})$
cancels. Multiplication with $\left(\prod_{j\in \JsetW{\srfactorvect}} \prod_{k=1}^{n_j} \frac{\eta_j}{\mu_j(k)}\right)^{-1}$ yields

\begin{eqnarray*}
&&\left(\sum_{j\in\JsetW{\srfactorvect}}\lambda \routmx^{(\avvect)}(0,j)
+\sum_{j\in\JsetW{\srfactorvect}}1_{[n_{j}>0]}\srfactor_j\mu_{j}(n_{j})
(1-\routmx^{(\avvect)}(j,j))\right)\nonumber \\
 & = & \sum_{i\in\JsetW{\srfactorvect}}
 \frac{\mu_i(n_i)}{\eta_i} 1_{[n_{i}>0]}\lambda\routmx^{(\avvect)}(0,i)\nonumber \\
 && +\sum_{j\in\JsetW{\srfactorvect}}
 \sum_{i\in\JsetW{\srfactorvect}
 \backslash\{j\}}\frac{\mu_i(n_i)}{\eta_i} 1_{[n_{i}>0]} \frac{\eta_j}{\mu_j(n_j+1)}
 \srfactor_j\mu_{j}(n_{j}+1)\routmx^{(\avvect)}(j,i)\nonumber \\
 &  & +\sum_{j\in\JsetW{\srfactorvect}} \frac{\eta_j}{\mu_j(n_j+1)}\srfactor_j\mu_{j}(n_{j}+1)
 \routmx^{(\avvect)}(j,0)\,.
\end{eqnarray*}
We replace $\srfactor_j$ by $\alpha_j$ which are equal for all $j \in \Jset$ and get
\begin{eqnarray}
&&\left(\sum_{j\in\JsetW{\srfactorvect}}\lambda \routmx^{(\avvect)}(0,j)
+\sum_{j\in\JsetW{\srfactorvect}}1_{[n_{j}>0]}\alpha_j\mu_{j}(n_{j})
(1-\routmx^{(\avvect)}(j,j))\right)\nonumber \\
 & = & \sum_{i\in\JsetW{\srfactorvect}}
 \frac{\mu_i(n_i)}{\eta_i} 1_{[n_{i}>0]}\lambda\routmx^{(\avvect)}(0,i)\nonumber \\
 && +\sum_{j\in\JsetW{\srfactorvect}}
 \sum_{i\in\JsetW{\srfactorvect}
 \backslash\{j\}}\frac{\mu_i(n_i)}{\eta_i} 1_{[n_{i}>0]} \frac{\eta_j}{\mu_j(n_j+1)}
 \alpha_j\mu_{j}(n_{j}+1)\routmx^{(\avvect)}(j,i)\nonumber \\
 &  & +\sum_{j\in\JsetW{\srfactorvect}} \frac{\eta_j}{\mu_j(n_j+1)}\alpha_j\mu_{j}(n_{j}+1)
 \routmx^{(\avvect)}(j,0)\,.\label{eq:RS-LS-RAND-X-alpha-gbe-reduced2}
\end{eqnarray}
Reordering and canceling this yields
\begin{eqnarray}
&&\left(\sum_{j\in\JsetW{\srfactorvect}}\lambda \routmx^{(\avvect)}(0,j)+\sum_{j\in\JsetW{\srfactorvect}}1_{[n_{j}>0]}
\alpha_j\mu_{j}(n_{j})
\right)\nonumber \\
 & = & \sum_{i\in\JsetW{\srfactorvect}}
 \frac{\mu_i(n_i)}{\eta_i} 1_{[n_{i}>0]}\lambda\routmx^{(\avvect)}(0,i)\nonumber \\
&&  +\sum_{j\in\JsetW{\srfactorvect}}
\sum_{i\in\JsetW{\srfactorvect}} \frac{\mu_i(n_i)}{\eta_i} 1_{[n_{i}>0]} \frac{\eta_j}{\mu_j(n_j+1)}
\alpha_j
 \mu_{j}(n_{j}+1)\routmx^{(\avvect)}(j,i)\nonumber \\
 &  & +\sum_{j\in\JsetW{\srfactorvect}} \frac{\eta_j}{\mu_j(n_j+1)}\alpha_j\mu_{j}(n_{j}+1)
 \routmx^{(\avvect)}(j,0)\,,\label{eq:RS-LS-RAND-X-alpha-gbe-reduced3}
\end{eqnarray}
and so
\begin{eqnarray}
&&\left(\sum_{j\in\JsetW{\srfactorvect}}\lambda \routmx^{(\avvect)}(0,j)
+\sum_{j\in\JsetW{\srfactorvect}}1_{[n_{j}>0]}
\alpha_j\mu_{j}(n_{j})
\right) \label{eq:RS-LS-RAND-X-alpha-gbe-reduced4}\\
 & = & \sum_{i\in\JsetW{\srfactorvect}}
 \frac{\mu_i(n_i)}{\eta_i} 1_{[n_{i}>0]}\lambda\routmx^{(\avvect)}(0,i)
  +\sum_{j\in\JsetW{\srfactorvect}}
  \sum_{i\in\JsetW{\srfactorvect}}
  \frac{\mu_i(n_i)}{\eta_i} 1_{[n_{i}>0]} \alpha_j{\eta_j}
 \routmx^{(\avvect)}(j,i)\nonumber \\
&& +\sum_{j\in\JsetW{\srfactorvect}} {\eta_j}\alpha_j
 \routmx^{(\avvect)}(j,0)\,.\nonumber
\end{eqnarray}
The first term on the left side and the last term on the right side equate
\begin{eqnarray*}
&&\sum_{j\in\JsetW{\srfactorvect}}\lambda \routmx^{(\avvect)}(0,j) = \lambda(1- \routmx^{(\avvect)}(0,0))\,,\quad \text{and}\\
&&
\underbrace{
\left(\sum_{j\in\JsetW{\srfactorvect}} {\eta_j}\alpha_j
 \routmx^{(\avvect)}(j,0) + \eta_0\alpha_0\routmx^{(\avvect)}(0,0)\right)}_{=\eta_0\alpha_0}
  - \eta_0\alpha_0\routmx^{(\avvect)}(0,0)=\lambda(1-\routmx^{(\avvect)}(0,0))\,,
\end{eqnarray*}
where we used $\eta_0 =\lambda$, $\alpha_0=1$,
and that $(\eta_j\alpha_j:j\in\Jset_0)$ is an invariant
measure for $\routmx^{(\avvect)}$.
So \eqref{eq:RS-LS-RAND-X-alpha-gbe-reduced4} reduces to
\begin{eqnarray}
&&\sum_{j\in\JsetW{\srfactorvect}}
1_{[n_{j}>0]}\alpha_j\mu_{j}(n_{j})
 \label{eq:RS-LS-RAND-X-alpha-gbe-reduced5}\\
 & = & \sum_{i\in\JsetW{\srfactorvect}}
 \frac{\mu_i(n_i)}{\eta_i} 1_{[n_{i}>0]}\lambda\routmx^{(\avvect)}(0,i)
  +\sum_{j\in\JsetW{\srfactorvect}}\sum_{i\in\JsetW{\srfactorvect}} \frac{\mu_i(n_i)}{\eta_i} 1_{[n_{i}>0]} \alpha_j{\eta_j}
 \routmx^{(\avvect)}(j,i)\,.\nonumber
\end{eqnarray}
Take any $i\in\JsetW{\srfactorvect}$ with ${n_{i}>0}$ and consider the summands with this $i$:
\begin{eqnarray*}
&&\alpha_i1_{[n_{i}>0]}\mu_{i}(n_{i})
 =  \frac{\mu_i(n_i)}{\eta_i} 1_{[n_{i}>0]}\lambda\routmx^{(\avvect)}(0,i)
  +\sum_{j\in\JsetW{\srfactorvect}} \frac{\mu_i(n_i)}{\eta_i} 1_{[n_{i}>0]} \alpha_j{\eta_j}
 \routmx^{(\avvect)}(j,i)\,,\nonumber
\end{eqnarray*}
which is
\begin{eqnarray*}
\alpha_i {\eta_i}
 &=&  \lambda \routmx^{(\avvect)}(0,i)
  +\sum_{j\in\JsetW{\srfactorvect}}  \alpha_j{\eta_j}
 \routmx^{(\avvect)}(j,i)\,,\nonumber
\end{eqnarray*}
and recalling $\eta_0 =\lambda$,
$\alpha_0=1$ and $\alpha_j=0$ for ${j\in\JsetB{\srfactorvect}}$ this is
\begin{eqnarray} \label{eq:RS-LS-RAND-X-alpha-gbe-reduced6}
&&\alpha_i {\eta_i}
 = \sum_{j\in\JsetW{\srfactorvect}\cup \{0\}}  \alpha_j{\eta_j}
 \routmx^{(\avvect)}(j,i)\,.
\end{eqnarray}
Note, that any $i$ will occur in such a procedure for
some suitable state vector with $n_i>0$.
Therefore, if \eqref{eq:RS-LS-RAND-X-alpha-gbe-reduced6} would be true, we would
eventually arrive at
\begin{eqnarray*}
&&(\alpha_j{\eta_j}:j\in\Jset_0)\cdot \routmx^{(\avvect)}
=(\alpha_j{\eta_j}:j\in\Jset_0)\,.\nonumber
\end{eqnarray*}
Now \eqref{eq:RS-LS-RAND-X-alpha-gbe-reduced6} is true for the setting of \prettyref{thm:RS-JN-GB-skip-degrading}
by \prettyref{thm:RS-r-alpha} and for the setting of
\prettyref{thm:RS-JN-GB-refl-degrading}
by Proposition \ref{prop:RS-reflection-eta-alpha}, which finishes the parts
for $\JsetB{\srfactorvect}\neq \emptyset$  of the proofs.

The case $\JsetB{\srfactorvect}=\emptyset$ is proved similarly.
\end{proof}



\subsubsection{Upgraded and/or degraded service and adapted routing}
\label{sect:RS-upgraded}
We allow now more general changes of service rates by non-negative factors
$\srfactor_j \in [0,\infty)$. This means that we either speed up service at node $j$, if
$\srfactor_j >1$ or have a degraded server at node $j$, if $\srfactor_j <1$.
The case $\srfactor_j \leq 1$  $\forall j\in\Jset$ was considered in the previous section.
When at least one service rate increases, i.e. when
$||\srfactorvect||_{\infty}>1$, we introduce a new mechanism to adapt the network's load and
routing.
In this case we increase the total network input by a factor $\tnfactor=||\srfactorvect||_{\infty}>1$
and choose as acceptance probability vector $\avvect$ the relative
service rate changes $\alpha_j=\frac{\srfactor_j}{||\srfactorvect||_{\infty}}$.

 We will proceed  with the general case $\srfactor_j \in [0,\infty)$
in a way that the previous case is covered by our general formalism.

Remark: If $\gamma_j>1$,  node $j$ can  process more load without being overloaded, 
which is easily seen
by considering a single $M/M/1/\infty$ queue. In a network however, this additional load departing from $j$ can cause overloaded other nodes. It will be therefore possible that some of the offered
new total input of rate $\beta\cdot \lambda$  will not be accepted after readjusting the routing.
Our randomized random walk algorithms form Section \ref{sect:RS-randomwalks} will automatically 
compute the rejection rates for the external arrivals at all nodes.

The new network process will be denoted as in Section \ref{sect:RS-degraded} as
$\X^{(\srfactor)}=(X^{(\srfactorvect)}(t):t\geq 0)$, the vector process recording the
queue lengths in the network.
$X^{(\srfactorvect)}_t=(X^{(\srfactorvect)}_1(t),\dots,X^{(\srfactorvect)}_J(t))\in \N^{\Jset}_0$ reads: at time $t$ there are
$X^{(\srfactorvect)}_j(t)$ customers present at node $j$, either in service or
waiting. The assumptions put on the system imply that $\X$ is a
strong Markov process on state space $\N^{\Jset}_0$  with generator
$Q^{\X^{(\srfactorvect)}} = : Q^{{(\srfactorvect)}} = (q^{(\srfactorvect)}(\mathbf{n},\mathbf{n}^{'}):\mathbf{n},\mathbf{n}^{'}\in\mathbb{N}_{0}^{\bar{J}})$.
The strict positive transition rates of $Q^{(\srfactorvect)}$ are
 under both rerouting regimes for $\mathbf{n}=(n_1,\dots,n_J)\in \mathbb{N}_{0}^{\bar{J}}$

The generator $Q^{(\srfactorvect)}$ of this process is
\begin{align}\label{eq:RS-GB-JN-general}
q^{(\srfactorvect)}(\mathbf{n}\qsep\mathbf{n}+\evect_{i}) &
 = \beta \lambda \routmx^{(\avvect)}(0,i),
 &i\in\Jset_0\, ,\\
q^{(\srfactorvect)}(\mathbf{n}\qsep\mathbf{n}-\evect_{j}+\evect_{i}) &
=  1_{[n_{j}>0]}\srfactor_j  \mu_{j}(n_{j})\routmx^{(\avvect)}(j,i)
& i,j \in\Jset,\, i\neq j,
\nonumber\\
q^{(\srfactorvect)}(\mathbf{n}\qsep\mathbf{n}-\evect_{j}) &
=  1_{[n_{j}>0]}\srfactor_j \mu_{j}(n_{j})\routmx^{(\avvect)}(j,0),
& j\in\Jset\,.\nonumber
\end{align}

\begin{theorem}
\label{thm:RS-JN-GB-skip-general}
{\sc [Modified Jackson networks: Randomized skipping and general change of service]}
Let $\X$ be an  ergodic  Jackson network process as described in Section \ref{sect:RS-jackson-networks}
with stationary distribution $\xi$ from \eqref{eq:RS-jackson-steadystate},
where the service intensities $\mu_{i}(n_{i})$  at node $i$ are changed by a factor
$\srfactor_i\in [0,\infty)$ for $i\in \Jset$.
Denote
\begin{eqnarray}
	\alpha_0&=&1\,, \nonumber \\
	\alpha_j &=&
	\begin{cases}
	\srfactor_j&
	 \qquad \text{\upshape if}\quad ||\srfactorvect||_{\infty} \leq 1 \,,\\
	\frac{\srfactor_j}{||\srfactorvect||_{\infty}} &
	 \qquad \text{\upshape if}\quad ||\srfactorvect||_{\infty} > 1 \,,
	\end{cases}
	\qquad \forall j \in \Jset \,,
	\label{eq:RS-general-availability}
\end{eqnarray}
and
\begin{equation}
\label{eq:RS-total-network-input-factor}
	\tnfactor :=
	\begin{cases}
	1
	& \qquad \text{\upshape if}\quad ||\srfactor||_{\infty} \leq 1 \,,\\
	||\srfactor||_{\infty}
	& \qquad \text{\upshape if}\quad ||\srfactor||_{\infty} > 1 \,,
	\end{cases}
\end{equation}
and change routing by randomized skipping with  $\avvect=(\alpha_i:i\in \Jset_0)$  according
to \prettyref{thm:RS-r-alpha}, and change the total network input  by factor $\tnfactor$.
 Then
 $\xi$ is a stationary distribution for $\X^{(\srfactorvect)}=(X^{(\srfactorvect)}(t):t\geq 0)$ as well.\\
Moreover, if  $\JsetB{\srfactorvect} = \emptyset$, then $\X^{(\srfactorvect)}$ is ergodic.\\
If  $\JsetB{\srfactorvect}\neq \emptyset$, then $\X^{(\srfactorvect)}$ is not irreducible on
$\mathbb{N}_{0}^{\Jset}$ and its state space is divided into an infinite set of closed subspaces
\begin{equation*}
\mathbb{N}_{0}^{\JsetW{\srfactorvect}}\times \{(n_j:j\in\JsetB{\srfactorvect})\}\quad\forall
(n_j:j\in\JsetB{\srfactorvect})\in \mathbb{N}_{0}^{\JsetB{\srfactorvect}}\,.
\end{equation*}
For any probability distribution $\varphi$ on
$\mathbb{N}_{0}^{\JsetB{\srfactorvect}}$
there exists a stationary distribution $\xi^{(\srfactorvect)}_\varphi$ for $\X^{(\srfactorvect)}$, which is  for
$\nvect=(n_1,\dots,n_J)\in \mathbb{N}_{0}^{\bar{J}}$
\begin{equation*} 
\xi^{(\srfactorvect)}_\varphi(\nvect)=\xi^{(\srfactorvect)}_\varphi(n_1,\dots,n_J) =
\prod_{j\in \JsetW{\srfactorvect}} \prod_{k=1}^{n_j} \frac{\eta_j}{\mu_j(k)} C(j)^{-1}\cdot
\varphi(n_j:j\in\JsetB{\srfactorvect})\,.
\end{equation*}
\end{theorem}

\begin{theorem}
\label{thm:RS-JN-GB-refl-general}
{\sc [Modified Jackson networks: Randomized reflection and general change of service]}
Let $\X$ be an  ergodic  Jackson network process as described in Section \ref{sect:RS-jackson-networks}
with reversible routing matrix $r$ and
with stationary distribution $\xi$ from \eqref{eq:RS-jackson-steadystate},
where the service intensities $\mu_{i}(n_{i})$  at node $i$ are changed by a factor
$\srfactor_i\in [0,\infty)$ for $i\in \Jset$.

Take $\avvect=(\alpha_i:i\in \Jset_0)$ and $\tnfactor$ as defined in \prettyref{thm:RS-JN-GB-skip-general}, and change routing by randomized reflection according
to Proposition \ref{prop:RS-reflection-eta-alpha} with  $\avvect=(\alpha_i:i\in \Jset_0)$,
and the total network input by factor $\tnfactor$.

Then
 $\xi$ is a stationary distribution for $\X^{(\srfactorvect)}=(X^{(\srfactorvect)}_t:t\geq 0)$ as well.\\
Moreover, if $\JsetB{\srfactorvect} = \emptyset$, then $\X^{(\srfactorvect)}$ is ergodic.\\
If $\JsetB{\srfactorvect}\neq \emptyset$, then $\X^{(\srfactorvect)}$ is not irreducible on
$\mathbb{N}_{0}^{\Jset}$ and its state space is divided into an infinite set of closed subspaces
\begin{equation*}
\mathbb{N}_{0}^{\JsetW{\srfactorvect}}\times \{(n_j:j\in\JsetB{\srfactorvect})\}\quad\forall
(n_j:j\in\JsetB{\srfactorvect})\in \mathbb{N}_{0}^{\JsetB{\srfactorvect}}\,.
\end{equation*}
For any probability distribution $\varphi$ on
$\mathbb{N}_{0}^{\JsetB{\srfactorvect}}$
there exists a stationary distribution $\xi^{(\srfactorvect)}_\varphi$ for $\X^{(\srfactorvect)}$, which is  for
$\mathbf{n}=(n_1,\dots,n_J)\in \mathbb{N}_{0}^{\bar{J}}$
\begin{equation*} 
\xi^{(\srfactorvect)}_\varphi(\mathbf{n})=\xi^{(\srfactorvect)}_\varphi(n_1,\dots,n_J) =
\prod_{j\in \JsetW{\srfactorvect}} \prod_{k=1}^{n_j} \frac{\eta_j}{\mu_j(k)} C(j)^{-1}\cdot
\varphi(n_j:j\in\JsetB{\srfactorvect})\,.
\end{equation*}
\end{theorem}
\begin{proof}[Proof
of \prettyref{thm:RS-JN-GB-skip-general} and  \prettyref{thm:RS-JN-GB-refl-general}]
The global balance equation ${x}\cdot Q^{(\srfactorvect)}=0$ for the joint queue length process
$X^{(\srfactorvect)}$ of the modified system is in both settings for $\mathbf{n}=(n_1,\dots,n_J) \in \N_0^{\Jset}$
\begin{eqnarray}
&  & x(\mathbf{n})\left(\sum_{j\in\Jset}\tnfactor\lambda \routmx^{(\avvect)}(0,j)+\sum_{j\in\Jset}1_{[n_{j}>0]}\srfactor_j\mu_{j}(n_{j})
(1-\routmx^{(\avvect)}(j,j))\right)\nonumber \\
 & = & \sum_{i\in\Jset} x(\mathbf{n}-\evect_{i})1_{[n_{i}>0]}\tnfactor\lambda\routmx^{(\avvect)}(0,i)
  +\sum_{j\in\Jset}\sum_{i\in\Jset\backslash\{j\}}x(\mathbf{n}-\evect_{i}+\evect_{j})1_{[n_{i}>0]}
 \srfactor_j\mu_{j}(n_{j}+1)\routmx^{(\avvect)}(j,i)\nonumber \\
 &  & +\sum_{j\in\Jset}x(\mathbf{n}+e_{j})\srfactor_j\mu_{j}(n_{j}+1)
 \routmx^{(\avvect)}(j,0)\,
 .\label{eq:RS-RAND-X-general-srfactor-gbe}
\end{eqnarray}

We first consider  the case $\JsetB{\srfactorvect}\neq \emptyset$. Then for $i\in \JsetB{\srfactorvect}$
we have $\alpha_i=0$ and $\routmx^{(\avvect)}(j,i) =0$ for all $j\in\Jset_0$  and \eqref{eq:RS-RAND-X-general-srfactor-gbe} reduces to
\begin{eqnarray}
&  & x(\mathbf{n})\left(\sum_{j\in W(\srfactorvect)}\tnfactor\lambda \routmx^{(\avvect)}(0,j)
+\sum_{j\in W(\srfactorvect)}1_{[n_{j}>0]}\srfactor_j\mu_{j}(n_{j})
(1-\routmx^{(\avvect)}(j,j))\right)\nonumber \\
 & = & \sum_{i\in W(\srfactorvect)} x(\mathbf{n}-\evect_{i})1_{[n_{i}>0]}
 \tnfactor\lambda\routmx^{(\avvect)}(0,i) \nonumber \\
 &  & +\sum_{j\in W(\srfactorvect)}\sum_{i\in W(\srfactorvect) \backslash\{j\}}x(\mathbf{n}-\evect_{i}+\evect_{j})1_{[n_{i}>0]}
 \srfactor_j\mu_{j}(n_{j}+1)\routmx^{(\avvect)}(j,i)\nonumber \\
 &  & +\sum_{j\in W(\srfactorvect)}x(\mathbf{n}+\evect_{j})\srfactor_j\mu_{j}(n_{j}+1)
 \routmx^{(\avvect)}(j,0)\,
 .\label{eq:RS-RAND-X-general-srfactor-gbe-reduced1}
\end{eqnarray}
Inserting $x(n_1,\dots,n_J) = \prod_{j\in \JsetW{\srfactorvect}} \prod_{k=1}^{n_j} \frac{\eta_j}{\mu_j(k)} C(j)^{-1}\cdot
\varphi(n_j:j\in\JsetB{\srfactorvect})$ for any probability density $\varphi$ on
$\mathbb{N}_{0}^{\JsetB{\srfactorvect}}$ we see that immediately
$\prod_{j\in \JsetW{\srfactorvect}} C(j)^{-1}\cdot \varphi(n_j:j\in\JsetB{\srfactorvect})$
cancels. Multiplication with $\left(\tnfactor\prod_{j\in \JsetW{\srfactorvect}} \prod_{k=1}^{n_j} \frac{\eta_j}{\mu_j(k)}\right)^{-1}$ yields
\begin{eqnarray}
&&\left(\sum_{j\in\JsetW{\srfactorvect}}\lambda \routmx^{(\avvect)}(0,j)+\sum_{j\in\JsetW{\srfactorvect}}1_{[n_{j}>0]}
\frac{\srfactor_j}{\beta}
\mu_{j}(n_{j})
(1-\routmx^{(\avvect)}(j,j))\right)\nonumber \\
 & = & \sum_{i\in\JsetW{\srfactorvect}}
 \frac{\mu_i(n_i)}{\eta_i} 1_{[n_{i}>0]}\lambda\routmx^{(\avvect)}(0,i)\nonumber \\
 && +\sum_{j\in\JsetW{\srfactorvect}}\sum_{i\in\JsetW{\srfactorvect}
 \backslash\{j\}}\frac{\mu_i(n_i)}{\eta_i} 1_{[n_{i}>0]} \frac{\eta_j}{\mu_j(n_j+1)}
 \frac{\srfactor_j}{\beta}\mu_{j}(n_{j}+1)\routmx^{(\avvect)}(j,i)\nonumber \\
 &  & +\sum_{j\in\JsetW{\srfactorvect}} \frac{\eta_j}{\mu_j(n_j+1)}\frac{\srfactor_j}{\beta} \mu_{j}(n_{j}+1)
 \routmx^{(\avvect)}(j,0)\,.
\end{eqnarray}

Using the fact that $\srfactor_j/\beta = \srfactor_j/||\srfactorvect||_{\infty}=\alpha_j$ for all $j\in \Jset$
we get the equation

\begin{eqnarray}
&&\left(\sum_{j\in\JsetW{\srfactorvect}}\lambda \routmx^{(\avvect)}(0,j)+\sum_{j\in\JsetW{\srfactorvect}}1_{[n_{j}>0]}
\alpha_j
\mu_{j}(n_{j})
(1-\routmx^{(\avvect)}(j,j))\right)\nonumber \\
 & = & \sum_{i\in\JsetW{\srfactorvect}}
 \frac{\mu_i(n_i)}{\eta_i} 1_{[n_{i}>0]}\lambda\routmx^{(\avvect)}(0,i)\nonumber \\
 && +\sum_{j\in\JsetW{\srfactorvect}}\sum_{i\in\JsetW{\srfactorvect}
 \backslash\{j\}}\frac{\mu_i(n_i)}{\eta_i} 1_{[n_{i}>0]} \frac{\eta_j}{\mu_j(n_j+1)}
 \alpha_j\mu_{j}(n_{j}+1)\routmx^{(\avvect)}(j,i)\nonumber \\
 &  & +\sum_{j\in\JsetW{\srfactorvect}} \frac{\eta_j}{\mu_j(n_j+1)}\alpha_j \mu_{j}(n_{j}+1)
 \routmx^{(\avvect)}(j,0)\,.
\end{eqnarray}

The last equation is exactly \eqref{eq:RS-LS-RAND-X-alpha-gbe-reduced2}. The rest of the  proof is as in the proof of
\prettyref{thm:RS-JN-GB-skip-degrading} and
\prettyref{thm:RS-JN-GB-refl-degrading}.
\end{proof}

\begin{cor}
\label{cor:RS-JN-GB-alpha-1}
If in the setting of \prettyref{thm:RS-JN-GB-skip-general} or of \prettyref{thm:RS-JN-GB-refl-general}
the Jackson network process $\X$ is ergodic with stationary distribution $\xi$ from \eqref{eq:RS-jackson-steadystate},
and if after modi\-fication all nodes are still working, possibly with degraded capacity, i.e.
$\JsetB{\srfactorvect} = \emptyset,$ then  in both cases
 $\X^{(\srfactorvect)}=(X^{(\srfactorvect)}_t:t\geq 0)$ is ergodic as well with
 $\xi$ as unique stationary and limiting distribution.
\end{cor}
\begin{cor}
\label{cor:RS-JN-GB-alpha-2}
If in the framework of \prettyref{thm:RS-JN-GB-skip-general}, respectively of
\prettyref{thm:RS-JN-GB-refl-general} we have $\routmx^{(\avvect)}(0,0)>0$ then the
{\em effective arrival rate} after modification is $\tnfactor\lambda(1-\routmx^{(\avvect)}(0,0))$.
\end{cor}
 The following result summarizes the content of  \prettyref{thm:RS-JN-GB-refl-general}
and \prettyref{thm:RS-JN-GB-skip-general} and extends both to an abstract framework.
\begin{cor}
\label{cor:RS-JN-GB-alpha-3}
{\sc [Modified Jackson networks: General change of routing and general change of service]}
Let $\X$ be an  ergodic  Jackson network process as described in  \prettyref{sect:RS-jackson-networks}
with stationary distribution $\xi$ from \eqref{eq:RS-jackson-steadystate},
where the service intensities $\mu_{i}(n_{i})$  at node $i$ are changed by a factor
$\srfactor_i\in [0,\infty) $ for $i\in \Jset$.
We change routing
to follow some matrix  $\routmx^{(\avvect)}$ with invariant measure
$y=(\alpha_j \eta_j:j\in \Jset_0)$ and increase the total network input by a factor $\tnfactor$, where $\alpha_0=1$, $\alpha_j$ and $\tnfactor$ are defined as in
\eqref{eq:RS-general-availability} and
\eqref{eq:RS-total-network-input-factor}.

 We denote the resulting Markovian state process
on $\mathbb{N}_{0}^{\Jset}$ by  $\X^{(\srfactorvect)}=(X^{(\srfactorvect)}(t):t\geq 0)$.

Then $\xi$ is a stationary distribution for $\X^{(\srfactorvect)}=(X^{(\srfactorvect)}(t):t\geq 0)$ as well.\\
Moreover, we define $\JsetB{\srfactorvect} = \{j\in \Jset: \srfactor_j=0\}$.

If $\JsetB{\srfactorvect} = \emptyset$, then $\X^{(\srfactorvect)}$ is ergodic.

If $\JsetB{\srfactorvect}\neq \emptyset$, then $\X^{(\srfactorvect)}$ is not irreducible on
$\mathbb{N}_{0}^{\bar{J}}$ and its state space is divided into an infinite set of closed subspaces
\begin{equation*}
\mathbb{N}_{0}^{\JsetW{\srfactorvect}}\times \{(n_j:j\in\JsetB{\srfactorvect})\}\quad\forall
(n_j:j\in\JsetB{\srfactorvect})\in \mathbb{N}_{0}^{\JsetB{\srfactorvect}}\,.
\end{equation*}
For any probability distribution $\varphi$ on
$\mathbb{N}_{0}^{\JsetB{\srfactorvect}}$
there exists a stationary distribution $\xi^{(\srfactorvect)}_\varphi$ for $\X^{(\srfactorvect)}$, which is  for
$\nvect=(n_1,\dots,n_J)\in \mathbb{N}_{0}^{\bar{J}}$
\begin{equation*} 
\xi^{(\srfactorvect)}_\varphi(\nvect))=\xi^{(\srfactorvect)}_\varphi(n_1,\dots,n_J) =
\prod_{j\in \JsetW{\srfactorvect}} \prod_{k=1}^{n_j} \frac{\eta_j}{\mu_j(k)} C(j)^{-1}\cdot
\varphi(n_j:j\in\JsetB{\srfactorvect})\,.
\end{equation*}
\end{cor}

\subsection{Jackson networks in a random environment}\label{sect:RS-JN-random-env}
We consider a classical Jackson network as described in \prettyref{sect:RS-jackson-networks} and assume that its development is influenced by the time varying status of its external environment.
On the other hand the network may trigger the environment to change its status.
We therefore have a twofold interacting
dynamic which is determined on one side by the environment as a
jump process $\Y=(Y(t):t\geq 0)$ in continuous  time, changes of which result in changes of the network's parameter,
and on the other side by the network process $\X=(X(t):t\geq 0)$ as jump process  
where some jumps of $\X$ enforce the environment to immediately react to this jump.
To be more  precise:

\label{page:RS-JN-random-env}
The state (status) of the environment is recorded in a countable environment state space $K$
and  whenever the environment at time $t$ is in state $Y(t)=k\in K$ it changes its status
to $m\in K$ with rate $\nu(k,m)$.

The network process $\X$ records the joint queue length vector, as in
\prettyref{sect:RS-jackson-networks}, and $X_j(t) =n_j\in \N_0$  is the local queue length at node
$j\in \Jset$. Whenever the environment is in state $k\in K$ and
 at node $j$ a customer is served and leaves the network, then this jump of the local
queue length triggers with probability $R_j(k,m)$  the environment to jump  immediately from state $k$
to $m\in K$.

Associated with each environment state $k\in K$ is a  vector $\srfactorvect(k)\in [0,\infty)^{\Jset}$
which determines the factor by which the service capacities at the nodes are
changed, similar to the vector  $\srfactorvect \in [0,\infty)^{\Jset}$ in  \prettyref{sect:RS-modified-jackson-networks}.
This results in a state dependent service rate $\mu_j(n_j,k)=\srfactor_j(k)\cdot \mu_j(n_j)$ if the
queue length at node $j$ is $n_j$ and the environment status is $k$.

The network reacts to the impact of the environment when this is in state $k$ by modifying the routing
according to different strategies, which we have described in Sections \ref{sect:RS-randomskip} and
\ref{sect:RS-randomreflec}
and possibly with admitting
more customers into the network. The latter part
of the control strategy is set in force
whenever in environment state $k$
there exist some $\srfactor_j(k)>1$.
In such state $k \in K$ the overall
arrival rate to the network is increased
by $\tnfactor(\srfactorvect(k))=||\srfactorvect(k)||_{\infty}$
from $\lambda$ to $\lambda \cdot \tnfactor(\srfactorvect(k))$
.

A schematic example of this kind of system is shown on
\prettyref{fig:RS-jackson-skipping-environement}.

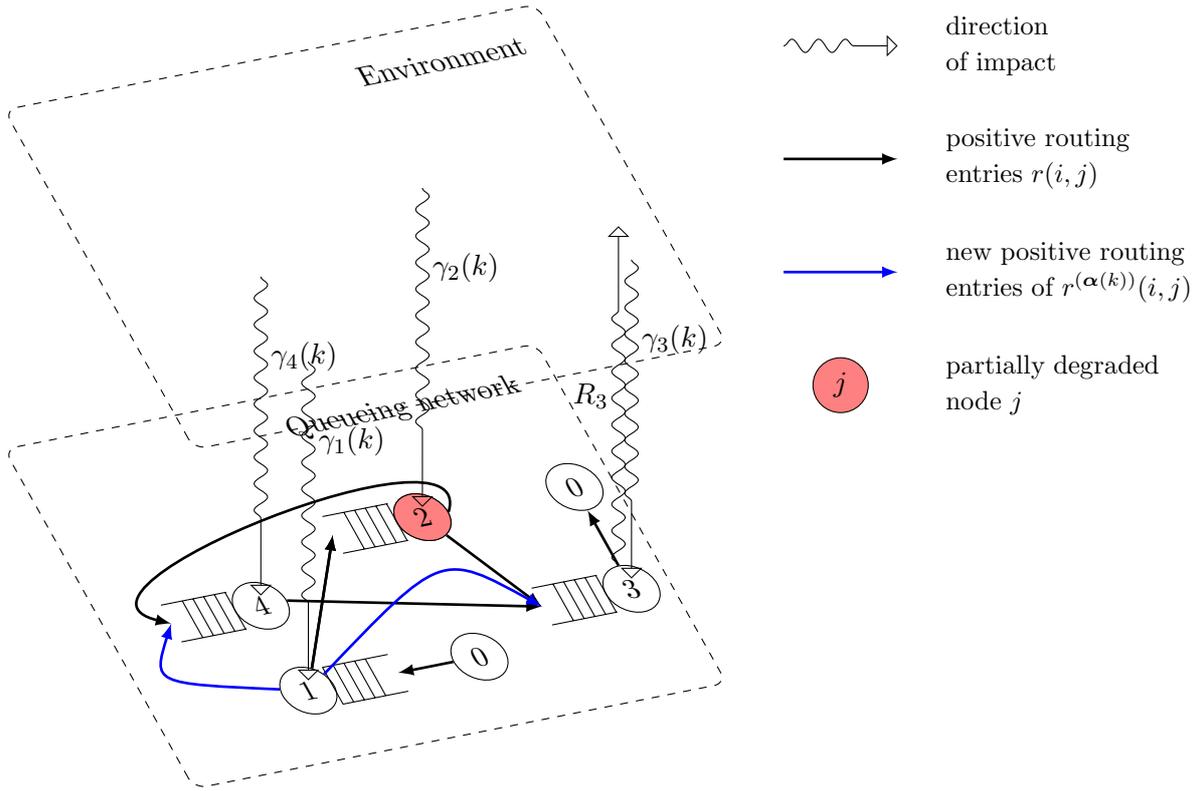
\begin{figure}[H]
\centering
\begin{tikzpicture}

\newcommand{\yslant}{0.2} 
\newcommand{\xslant}{-0.5}
\newcommand{\height}{6}

\begin{scope}


\end{scope}

\begin{scope}
[every node/.append style={yslant=\yslant,xslant=\xslant},yslant=\yslant,xslant=\xslant]
  \node()[] at (3.0,4) {Queueing network};
  \draw[dashed,rounded corners=1ex] (-2,-0.5) rectangle (5,4.5);

  \node(0-1-orgn)[] at (2.25,0.5) {};
  \node(0-2-orgn)[] at (4.5,2.5) {};
  \node(1-orgn)[] at (0,0.5) {};
  \node(1-Q-orgn)[] at (0.75,0.5) {};
  \node(2-orgn)[] at (2.5,2.5) {};
  \node(2-Q-orgn)[] at (1.75,2.5) {};
  \node(2-orgn-under)[] at (2.5,2) {};
  \node(2-orgn-1)[] at (3,2.5) {}; 
  \node(2-orgn-2)[] at (3,3) {};
  \node(2-orgn-3)[] at (0,3) {};
  \node(3-orgn)[] at (4.5,1) {};
  \node(3-Q-orgn)[] at (3.75,1) {};
  \node(4-orgn)[] at (0,1.75) {};
  \node(4-Q-orgn)[] at (-0.75,1.75) {};
  \node(4-orgn-under)[] at (-1.8,1) {};
  \node(5-orgn)[] at (0,1.75) {};

  \node(0-1)[draw, shape=circle, color=external.fg.color] at (0-1-orgn) {$0$};
  \node(0-2)[draw, shape=circle, color=external.fg.color] at (0-2-orgn) {$0$};
  \node(1)[draw, shape=circle] at (1-orgn) {$1$};
  \node(2-white)[draw, shape=circle, fill=white] at (2-orgn) {$2$};
  \node(2)[draw, shape=circle, fill=down.partially.bg.color] at (2-orgn) {$2$};
  \node(3)[draw, shape=circle] at (3-orgn) {$3$};
  \node(4)[draw, shape=circle] at (4-orgn) {$4$};

\node(1-Q)[draw, shape=queue, queue head=west, queue size=infinite,
minimum width =24 pt, minimum height=15pt] at(1-Q-orgn) {};
\node(2-Q)[draw, shape=queue, queue head=east, queue size=infinite,
minimum width =24 pt, minimum height=15pt] at(2-Q-orgn) {};
\node(3-Q)[draw, shape=queue, queue head=east, queue size=infinite,
minimum width =24 pt, minimum height=15pt] at(3-Q-orgn) {};
\node(4-Q)[draw, shape=queue, queue head=east, queue size=infinite,
minimum width =24 pt, minimum height=15pt] at(4-Q-orgn) {};

\draw[arrows={-latex}, thick,
  line width=\customerslinewidth] (1) -- (2-Q.west);
\draw[arrows={-latex}, thick,
  line width=\customerslinewidth] (2) -- (3-Q.west);
\draw[arrows={-latex}, thick,
  line width=\customerslinewidth] (3) -- (0-2);
\draw[arrows={-latex}, thick,
  line width=\customerslinewidth] (0-1) -- (1-Q.east);
\draw[arrows={-latex}, thick,
  line width=\customerslinewidth] (2.east) to [bend right=140] (4-Q.west);
\draw[arrows={-latex}, thick,
  line width=\customerslinewidth] (4.east) -- (3-Q.west);

\draw[arrows={-latex}, thick,
color=blue,
line width=\customerslinewidth] (1) .. controls (2-orgn-under) .. (3-Q.west);

\draw[arrows={-latex}, thick,
color=blue,
line width=\customerslinewidth,
bend right=40] (1) .. controls (4-orgn-under) ..(4-Q.west);

\end{scope}

\begin{scope}
[shift={(0,4.5)}, every node/.append style={yslant=\yslant,xslant=\xslant},yslant=\yslant,xslant=\xslant]

  \node()[] at (3.5,4) {Environment};
  \draw[dashed,rounded corners=1ex] (-2,-0.5) rectangle (5,4.5);

  \node(1-orgn-env)[] at (0,0.5) {};
  \node(2-orgn-env)[] at (2.5,2.5) {};
  \node(3-orgn-env)[] at (4.5,1) {};
  \node(4-orgn-env)[] at (0,1.75) {};
  \node(5-orgn)[] at (0,1.75) {};

  \node(1-env)[draw, shape=circle,opacity=.0] at (1-orgn-env) {$1$};
  \node(2-env)[draw, shape=circle,opacity=.0,
    fill=  down.partially.bg.color] at (2-orgn-env) {$2$};
  \node(3-env)[draw, shape=circle,opacity=.0] at (3-orgn-env) {$3$};
  \node(4-env)[draw, shape=circle,opacity=.0] at (4-orgn-env) {$4$};

\end{scope}

\begin{scope}[every edge/.append
style={decorate,
decoration={snake,post length=1cm}}
]


\path[arrows={-open triangle 90}]
     (1-orgn-env) edge[decorate, decoration={snake}] node[near start, right] {$\srfactor_1(k)$}  (1-orgn);
\path[arrows={-open triangle 90}]
     (2-orgn-env) edge node[near start, right] {$\srfactor_2(k)$}  (2-orgn);



\path[arrows={-open triangle 90}]
     (3-orgn-env) edge node[near start, right] {$\srfactor_3(k)$}  (3-orgn);
\path[arrows={-open triangle 90}]
     (4-orgn-env) edge node[near start, right] {$\srfactor_4(k)$}  (4-orgn);


\path[arrows={-open triangle 90}]
     (3.north) edge node[left] {$R_3$}  (3-env.north);
\draw[arrows={-latex}, thick,
  line width=\customerslinewidth] (1) -- (2-Q.west);

\end{scope}

\begin{scope}[shift={(6,9)}]
\path[arrows={-open triangle 90}]
     (0,0)edge[decorate, decoration={snake,post length=0.5cm}]  (1.5,0);
\node()[right, align=left] at (2.0,0){\small direction\\\small of impact};

\draw[arrows={-latex}, thick,
  line width=\customerslinewidth] (0,-1.5) -- (1.5,-1.5);
\node()[right, align=left] at (2.0,-1.5){\small positive routing\\\small entries $r(i,j)$};

\draw[arrows={-latex}, thick,
  color=blue,
  line width=\customerslinewidth] (0,-3) -- (1.5,-3);
\node()[right, align=left] at (2.0,-3){\small new positive routing\\\small entries of $r^{(\avvect(k))}(i,j)$};

\node()[draw, shape=circle, fill=down.partially.bg.color] at (0.75,-4.5) {$j$};
\node()[right, align=left] at (2.0,-4.5){\small partially degraded\\\small node $j$ };


\end{scope}
\end{tikzpicture}
\caption{Jackson network and random environment with
two-way interaction, The service rates are
modified by the environment. The routing is adopted
according to the skipping rule. The customers, who
leave the node 3, modify the environment.}
\label{fig:RS-jackson-skipping-environement}
\end{figure}

\subsubsection{Rerouting according to randomized skipping}\label{sect:RS-JN-random-env-skip}
In this section we consider the case that the modification of routing in reaction to the
servers' change of capacity is by randomized skipping according to Section \ref{sect:RS-randomskip}.
We will investigate this case in
full detail, while other modifications thereafter  can be described with less details.\\
 We need here environment dependent rerouting  with acceptance probabilities
$\avvect = \avvect(\srfactor(k))$, modified rerouting matrices $r^{(\avvect(\srfactor(k)))}$,
and overall load factors  $\tnfactor(\srfactor(k))$.

To keep notation short we will write $\avvect(k) =(\alpha_j(k):j\in \Jset_0)$,
instead of $\avvect(\srfactor(k))$,
$r^{(\avvect(k))}$ instead of $r^{(\avvect(\srfactor(k)))}$
and $\tnfactor(k)$ instead of $\tnfactor(\srfactor(k))$.

The randomized skipping according to \prettyref{sect:RS-randomskip} yields a routing
regime $r^{(\avvect(k))}$ according to \prettyref{thm:RS-r-alpha},
and the total service input rate
is changed by a factor $\tnfactor(k)$.
$\avvect$ and $\tnfactor$ are defined
 similar to  \eqref{eq:RS-general-availability}
and \eqref{eq:RS-total-network-input-factor},
 i.e. for  $k\in K$:
\begin{eqnarray}
	\alpha_0(k)&=&1\,, \nonumber \\
	\alpha_j(k) &=&
	\begin{cases}
	\srfactor_j(k)
	& \qquad \text{\upshape if}\quad
	||\srfactorvect(k)||_{\infty} \leq 1\,, \\
	\frac{\srfactor_j(k)}{||\srfactorvect(k)||_{\infty}}
	& \qquad \text{\upshape if}\quad
	||\srfactorvect(k)||_{\infty} > 1\,,
	\end{cases}
	\qquad \forall j \in \Jset
	\label{eq:RS-general-availability-from-k}
\end{eqnarray}
and
\begin{equation}
\label{eq:RS-total-network-input-factor-from-k}
	\tnfactor(k) :=
	\begin{cases}
	1
	& \qquad \text{\upshape if}\quad
	||\srfactor(k)||_{\infty} \leq 1 \,, \\
	||\srfactor(k)||_{\infty}
	& \qquad \text{\upshape if}\quad ||\srfactor(k)||_{\infty} > 1\,.
	\end{cases}
\end{equation}

We further define $\JsetB{\srfactorvect(k)}$ and $\JsetW{\srfactorvect(k)}$  similar to \prettyref{def:RS-JBA}
as set of completely broken down nodes, resp. as set of nodes which, although possibly being degraded
or upgraded, can still serve customers under environment condition $k$.\\

\begin{defn}
\label{def:RS-LS-RAND} We denote the coupled process {\sc (queue lengths, environment)} over time by
$\Z=(\X,\Y) = (Z(t):t\geq 0) = ((X(t),Y(t)):t\geq 0)$ with state space $E:= \N_0^{\Jset}\times K$.
\end{defn}

As described above the dynamics of $\Z$ relies for the environment process $\Y$ especially on
a generator matrix $\myV=(\nu(k,m): k,m\in K)$ and  stochastic matrices
$\envmx_j=(\envmx_j(k,m): k,m\in K), j\in \Jset$. Recall that the original extended routing matrix $r=(r(i,j):i,j\in\Jset_0)$ is irreducible
and that $r^{(\alpha(k))}$ in the considered case of randomized skipping  is irreducible on $\JsetW{\srfactor(k)}\cup \{0\}$.

With the standard independence assumptions for the
inter-arrival and service times and the conditional
independence assumptions for the routing process and the jumps of the environment triggered by
departing customers the following statement is obvious.

\begin{prop}\label{prop:RS-processZ}
The queue lengths-environment process $\Z$ is a homogeneous Markov process  on state space
$E:= \N_0^{\Jset}\times K$  with generator
$Q^\Z = (q^\Z((\mathbf{n},k),(\mathbf{n}^{'},k')):
(\mathbf{n},k),(\mathbf{n}^{'},k')\in E)$.
The strict positive transition rates of $Q^\Z$, which describe the physical actions of the network
and the environment, are for $(\mathbf{n},k)=((n_1,\dots,n_J),k)\in \mathbb{N}_{0}^{\Jset}\times K$
\begin{align}
q^\Z((\nvect,k),(\nvect+\evect_{i},k)) & = \tnfactor(k)\lambda \routmx^{(\avvect(k))}(0,i), & i\in \Jset\,,
\nonumber \\
q^\Z((\nvect,k),(\nvect-\evect_{j}+\evect_{i},k)) &
= 1_{[n_{j}>0]}
\srfactor_j(k)\mu_{j}(n_{j})\routmx^{(\avvect(k))}(j,i),& j,i\in \Jset,~ i\neq j\,, \nonumber \\
q^\Z((\nvect,k),(\nvect-\evect_{j},m)) &
 = 1_{[n_{j}>0]}\srfactor_j(k)\mu_{j}(n_{j})
\routmx^{(\avvect(k))}(j,0)R_j(k,m),& j\in\Jset,\nonumber \\
q^\Z((\nvect,k),(\nvect,m)) & = \nu(k,m),& m\in K\,.
\label{eq:RS-LS-RAND-q-transition-rates}
\end{align}
\end{prop}

\begin{theorem}
\label{thm:GS-BOUNDARY} Assume the queue lengths-environment process $\Z=(\X,\Y)$ from \prettyref{def:RS-LS-RAND}.
to be ergodic and assume further
that the pure Jackson network process $\X$ without environment is ergodic with stationary
and limiting distribution $\xi$ on $ \N_0^{\Jset}$  from \eqref{eq:RS-jackson-steadystate}
\begin{equation*}                        
\xi(\nvect)=\xi(n_1,\dots,n_J)
= \prod_{j=1}^{J} \prod_{k=1}^{n_j} \frac{\eta_j}{\mu_j(k)} C(j)^{-1},
\quad \nvect\in \N_0^{\Jset}\,,
\end{equation*}
with normalizing  constants $C(j)$ for the  marginal (over nodes) distributions of $\X$.\\
Then the queue lengths-environment process $\Z$  has the unique
 steady state distribution $\pi=(\pi(\nvect,k):\nvect\in\mathbb{N}_{0}^{\Jset},k\in K)$
of product form:
\[
\pi(\nvect,k)=\xi(\nvect)\theta(k),\quad \nvect\in\mathbb{N}_{0}^{\Jset},k\in K\,,
\]
where  $\theta$ is the unique stochastic solution of the following {\em reduced generator
equation} \\$\theta\cdot{Q}_{red}=0$ with

\begin{equation}\label{eq:RS-Q-red2}
{Q}_{red} := {\left[\myV
+ \sum_{j\in\Jset} \eta_j I_{(\srfactor_j  \bullet \routmx^{(\avvect(\cdot))}(j,0))}(\envmx_j - I)\right]}\,.
\end{equation}
Here we denote for $j\in \Jset$ the real valued functions
$\srfactor_j$ and
$\routmx^{(\avvect(\cdot))}(j,0)$ on $K$,
and by\\
$\srfactor_j \bullet \routmx^{(\avvect(\cdot))}(j,0)$ their point wise multiplication, the result of which we
interpret as vector to obtain the diagonal matrix $I_{(\srfactor_j \bullet \routmx^{(\avvect(\cdot))}(j,0))}$.

 \end{theorem}
\begin{proof}
The global balance equation of the process $\Z$ is
\begin{eqnarray*}
 &  & \pi(\nvect,k)\Bigg(\sum_{i\in\Jset}\tnfactor(k)\lambda \routmx^{(\avvect(k))}(0,i)
 +\underbrace{\sum_{m\in K\backslash\{k\}}\nu(k,m)}_{-\nu(k,k)}
 +\sum_{j\in\Jset}1_{[n_{j}>0]}\srfactor_j(k)\mu_{j}(n_{j})(1-\routmx^{(\avvect(k))}(j,j))\Bigg)
 \\
 & = & \sum_{i\in\Jset}\pi(\nvect-\evect_{i},k)1_{[n_{i}>0]}\tnfactor(k)\lambda \routmx^{(\avvect(k))}(0,i))\\
 &  & +\sum_{i\in\Jset}\sum_{j\in\Jset\backslash\{i\}} \pi(\nvect-\evect_{i}+\evect_{j},k)1_{[n_{i}>0]}\srfactor_j(k)\mu_{j}(n_{j}+1)\routmx^{(\avvect(k))}(j,i)\\
 &  & +\sum_{j\in\Jset}\sum_{m\in K}\pi(\nvect+\evect_{j},m)\srfactor_j(m)\mu_{j}(n_{j}+1)\routmx^{(\avvect(m))}(j,0)\envmx_j(m,k)\\
 &  & +\sum_{m\in K\backslash\{k\}}\pi(\nvect,m)\nu(m,k)
\end{eqnarray*}
Inserting $\pi(\nvect,k)=\xi(\nvect)\theta(k)$ and adding
$\xi(\nvect)\theta(k)\left((\nu(k,k) +
\sum_{j\in\Jset} 1_{n_j>0)} \gamma_j(k) \mu_j(n_j)\routmx^{(\avvect(k))}(j,j)\right)$
on both sides we obtain
\begin{eqnarray*}
 &  & \xi(\nvect)\theta(k)\left(\sum_{i\in\Jset}\tnfactor(k)\lambda \routmx^{(\avvect(k))}(0,i)
+\sum_{j\in\Jset}1_{[n_{j}>0]}\srfactor_j(k)\mu_{j}(n_{j})\right)
 \\
 & = & \sum_{i\in\Jset}\xi(\mathbf{n}-\evect_{i})\theta(k)1_{[n_{i}>0]}\tnfactor(k)\lambda r^{(\avvect(k))}(0,i))\\
 &  & +\sum_{i\in\Jset}\sum_{j\in\Jset} \xi(\nvect-\evect_{i}+\evect_{j})\theta(k)1_{[n_{i}>0]}
 \srfactor_j(k)\mu_{j}(n_{j}+1)\routmx^{(\avvect(k))}(j,i)\\
 &  & +\sum_{j\in\Jset}\sum_{m\in K}
\xi(\nvect+\evect_{j})\theta(m)\srfactor_j(m)\mu_{j}(n_{j}+1)\routmx^{(\avvect(m))}(j,0)\envmx_j(m,k)\\
 &  & +\sum_{m\in K}\xi(\nvect)\theta(m)\nu(m,k) \,.
\end{eqnarray*}
Rearranging terms and blowing up this is
\begin{eqnarray}
 &  & \theta(k)\left[\xi(\nvect)\left(\sum_{i\in\Jset}
 \tnfactor(k)\lambda \routmx^{(\avvect(k))}(0,i)
 +\sum_{j\in\Jset}1_{[n_{j}>0]}\srfactor_j(k)\mu_{j}(n_{j})\right)\right]\label{eq:GS-BOUNDARY-1}
 \\
 & = &\theta(k)
 \left[\sum_{i\in\Jset}\xi(\nvect-\evect_{i})
 1_{[n_{i}>0]}\tnfactor(k) \lambda \routmx^{(\avvect(k))}(0,i))\right.
 \nonumber\\
 &  &\quad\quad   + \sum_{i\in\Jset}\sum_{j\in\Jset} \xi(\nvect-\evect_{i}+\evect_{j})
 1_{[n_{i}>0]}\srfactor_j(k)\mu_{j}(n_{j}+1)\routmx^{(\avvect(k))}(j,i)\nonumber\\
 &  &\quad\quad  \left. +\sum_{j\in\Jset}\xi(\nvect+\evect_{j})
 \srfactor_j(k)\mu_{j}(n_{j}+1)
 \routmx^{(\avvect(k))}(j,0)\right]\nonumber\\
 &  & -\theta(k)\sum_{j\in\Jset}\xi(\nvect+\evect_{j})
 \srfactor_j(k)\mu_{j}(n_{j}+1)
 \routmx^{(\avvect(k))}(j,0)\nonumber\\
 &  & +\sum_{j\in\Jset}\sum_{m\in K}
 \xi(\nvect+\evect_{j})\theta(m)
 \srfactor_j(m)\mu_{j}(n_{j}+1)
 \routmx^{(\avvect(m))}(j,0)\envmx_j(m,k)
 \nonumber\\
 &  & +\sum_{m\in K}\xi(\mathbf{n})\theta(m)\nu(m,k)\nonumber
 \,.
\end{eqnarray}
For each fixed environment state $k$ the terms in squared brackets equate from \prettyref{thm:RS-JN-GB-skip-general}, where for $\JsetB{\srfactorvect(k)}$ we set in modified
notation $(\varphi\to\varphi(k))$ from that theorem the specific probabilities
\begin{equation}\label{eq:RS-varphi-k}
\varphi(k)(n_j:j\in\JsetB{\srfactorvect(k)}) :=
\prod_{j\in \JsetB{\srfactorvect(k)}} \prod_{k=1}^{n_j} \frac{\eta_j}{\mu_j(k)} C(j)^{-1},\quad
(n_j:j\in\JsetB{\srfactorvect(k)})\in\N_0^{\JsetB{\srfactorvect(k)}}\,.
\end{equation}
Dividing by $\xi(\nvect)$ and canceling $\mu_{j}(n_{j}+1)$ we arrive at
\begin{eqnarray*}
0 & = & -\theta(k)\sum_{j\in\Jset}\eta_j\srfactor_j(k)\routmx^{(\avvect(k))}(j,0)\\
 &  & +\sum_{j\in\Jset}\sum_{m\in K}\theta(m)\eta_j\srfactor_j(m)\routmx^{(\avvect(m))}(j,0)\envmx_j(m,k)\\
 &  & +\sum_{m\in K}\theta(m)\nu(m,k)\,.
\end{eqnarray*}
Rearranging terms we see
\begin{eqnarray*}
 &  & \theta(k)\sum_{j\in\Jset}\eta_j \srfactor_j(k) \routmx^{(\avvect(k))}(j,0)\\
 & = &\sum_{m\in K}\theta(m)\left(\nu(m,k) + \sum_{j\in\Jset} \eta_j \srfactor_j(m) \routmx^{(\avvect(m))}(j,0)\envmx_j(m,k) \right)\,,
\end{eqnarray*}
and
\begin{eqnarray*}
 &  & \theta(k)\left(-\nu(k,k) + \sum_{j\in\Jset}\eta_j\srfactor_j(k)
 \routmx^{(\avvect(k))}(j,0)(1 -
\envmx_j(k,k))
\right)\\
 & = &\sum_{m\in K\setminus\{k\}}\theta(m)\left(\nu(m,k) +
  \sum_{j\in\Jset} \eta_j\srfactor_j(m)\routmx^{(\avvect(m))}(j,0)\envmx_j(m,k) \right),
\end{eqnarray*}
which finally leads for any prescribed $k\in K$ to
\begin{equation}\label{eq:RS-Q-red1}
0  = \sum_{m\in K}\theta(m)\left(\nu(m,k) +
  \sum_{j\in\Jset} \eta_j\srfactor_j(m)\routmx^{(\avvect(m))}(j,0)(\envmx_j(m,k)-\delta_{mk}) \right)
  \,.
\end{equation}

Denote for $j\in \Jset$ the real valued functions $\srfactor_j$ and $\routmx^{(\avvect(\cdot))}(j,0)$ on $K$,
and by
$\srfactor_j \bullet \routmx^{(\avvect(\cdot))}(j,0)$ the point wise multiplication the result of which we
interpret as vector to obtain the diagonal matrix $I_{(\srfactor_j
\bullet \routmx^{(\avvect(\cdot))}(j,0))}$.
Then \eqref{eq:RS-Q-red1} can be written in matrix form as

\begin{equation}\label{eq:RS-Q-red3}
0  = \theta \underbrace{\left[\myV
+ \sum_{j\in\Jset} \eta_j I_{(\srfactor_j \bullet \routmx^{(\avvect(\cdot))}(j,0))}(\envmx_j - I)\right]}_{=: Q_{red}}\,.
\end{equation}

So we have identified \eqref{eq:RS-Q-red2}
and because the non diagonal elements of the matrix $Q_{red}$ on the right are non-negative whereas the
row sum is zero, $Q_{red}$ is a generator matrix of some Markov process.

If the equation \eqref{eq:RS-Q-red3} has no stochastic solution the global balance equation of $\Z$
would have a non-trivial non-negative solution  which  cannot be normalized. This would
contradict  ergodicity. The same argument
shows that the solution of \eqref{eq:RS-Q-red3} must be unique.
\end{proof}

Remark: Although in the original Jackson network the service (rate $\mu_j(n_j)$) and 
routing (probabilities $\routmx(i,j)$) are locally determined with respect to the transition graph of
$\routmx$, the network control may in general be by global algorithms due to the applied 
randomized skipping by $\routmx^{(\avvect)}$.

\subsubsection{Rerouting according to randomized reflection}\label{sect:RS-JN-random-env-refl}
In this section we assume  that the modification of routing in reaction to  the  servers'
change of capacities is by randomized reflection according to  \prettyref{sect:RS-randomreflec}, which yields a routing
regime $r^{(\avvect(k))}$ according to  \prettyref{prop:RS-reflection-eta-alpha}.
We use
$\avvect(k)$ and $\beta(k)$ as defined in
\eqref{eq:RS-general-availability-from-k}
and
\eqref{eq:RS-total-network-input-factor-from-k},
and take $\JsetB{\srfactorvect(k)}$ and $\JsetW{\srfactorvect(k)}$ as in \prettyref{def:RS-JBA}.

Recall, that the dynamics of the environment process $\Y$  is driven by 
a generator matrix $\myV=(\nu(k,m): k,m\in K)$ and  stochastic matrices
$\envmx_j=(\envmx_j(k,m): k,m\in K), j\in \Jset$  as described on p. \pageref{page:RS-JN-random-env}. Note, that the original extended routing matrix $r=(r(i,j):i,j\in\Jset_0)$ is irreducible
but under randomized reflection $r^{(\avvect(k))}$ may be reducible even on $\JsetW{\srfactorvect(k)}\cup \{0\}$,  which does not destroy the ergodicity of the system process $\Z=(\X,\Y)$. We then have a formally similar statement as in \prettyref{prop:RS-processZ}.

\begin{prop}\label{prop:RS-processZ-refl}
The queue lengths-environment process $\Z=(\X,\Y) = (Z(t):t\geq 0) = ((X(t),Y(t)):t\geq 0)$
is a homogeneous Markov process  on state space
$E:= \N_0^{\Jset}\times K$  with generator
$Q^\Z = (q^\Z((\mathbf{n},k),(\mathbf{n}^{'},k')):
(\mathbf{n},k),(\mathbf{n}^{'},k')\in E)$.
The strict positive transition rates of $Q^\Z$, are for $(\mathbf{n},k)=((n_1,\dots,n_J),k)\in \mathbb{N}_{0}^{\Jset}\times K$
\begin{align*}
q((\nvect,k),(\nvect+\evect_{i},k)) & = \tnfactor(k)\lambda \routmx^{(\avvect(k))}(0,i), & i\in \Jset\,,
\nonumber \\
q((\nvect,k),(\nvect-\evect_{j}+\evect_{i},k)) &
= 1_{[n_{j}>0]}
\srfactor_j(k)\mu_{j}(n_{j})\routmx^{(\avvect(k))}(j,i),& j,i\in \Jset,~ i\neq j\,, \nonumber \\
q((\nvect,k),(\nvect-\evect_{j},m)) &
 = 1_{[n_{j}>0]}\srfactor_j(k)\mu_{j}(n_{j})
\routmx^{(\avvect(k))}(j,0)R_j(k,m),& j\in\Jset,\nonumber \\
q((\nvect,k),(\nvect,m)) & = \nu(k,m),& m\in K\,.
\end{align*}
\end{prop}

As was pointed out in \prettyref{sect:RS-randomreflec} a necessary ingredient for successfully
applying randomized reflection as rerouting regime is reversibility of $r=(r(i,j):i,j\in\Jset_0)$,
which we will now set  in force.

\begin{theorem}
\label{thm:GS-BOUNDARY-refl} Consider the queue lengths-environment process $\Z=(\X,\Y)$ from
 \prettyref{prop:RS-processZ-refl} and assume that the extended routing matrix
$r=(r(i,j):i,j\in\Jset_0)$ is reversible for $\eta=(\eta_j:j\in\Jset_0)$.

Assume $\Z$ to be ergodic and assume further
that the pure Jackson network process $\X$ without environment is ergodic with stationary
and limiting distribution $\xi$ on $ \N_0^{\Jset}$  from \eqref{eq:RS-jackson-steadystate}
\begin{equation*}                        
\xi(\nvect)=\xi(n_1,\dots,n_J)
= \prod_{j=1}^{J} \prod_{k=1}^{n_j} \frac{\eta_j}{\mu_j(k)} C(j)^{-1},
\quad \nvect\in \N_0^{\Jset}\,.
\end{equation*}
Then the queue lengths-environment process $\Z$  has the unique
 steady state distribution $\pi=(\pi(\mathbf{n},k):\mathbf{n}\in\mathbb{N}_{0}^{\Jset},k\in K)$
of product form:
\[
\pi(\nvect,k)=\xi(\nvect)\theta(k),\quad \nvect\in\mathbb{N}_{0}^{\Jset},k\in K\,,
\]
where  $\theta$ is the unique stochastic solution of the following {\em reduced generator
equation} \\$\theta\cdot{Q}_{red}=0$ with

\begin{equation}\label{eq:RS-Q-red-rand-refl}
{Q}_{red} := {\left[\myV
+ \sum_{j\in\Jset} \eta_j I_{(\srfactor_j  \bullet \routmx^{(\avvect(\cdot))}(j,0))}(\envmx_j - I)\right]}\,.
\end{equation}
Here we denote for $j\in \Jset$ the real valued functions
$\srfactor_j$ and
$\routmx^{(\avvect(\cdot))}(j,0)$ on $K$,
and by\\
$\srfactor_j \bullet \routmx^{(\avvect(\cdot))}(j,0)$ their point wise multiplication, the result of which we
interprete as vector to obtain the diagonal matrix $I_{(\srfactor_j \bullet \routmx^{(\avvect(\cdot))}(j,0))}$.

\end{theorem}

The proof of the theorem is along the lines of the proof of \prettyref{thm:GS-BOUNDARY}, where we
used almost completely the general abstract notation  $r^{(\avvect(k))}$ for the rerouting regime.
Only when manipulating \eqref{eq:GS-BOUNDARY-1} we had to refer to the specific
\prettyref{thm:RS-JN-GB-skip-general},
which would be substituted now by referring to \prettyref{thm:RS-JN-GB-refl-general}. The residual derivations are similar.

Remark: The  control of the customers' routing under randomized reflection by $\routmx^{(\avvect)}$ 
is by local decisions with respect to the transition graph of $\routmx$. So the network process is locally determined as well.

\subsubsection{Rerouting by general randomization}\label{sect:RS-JN-random-env-general}

This last observation of the previous section clearly suggests to extract the general principle for randomized rerouting. We consider
modification of routing in reaction to
the  servers' change of capacities by environment dependent factors
 $\srfactorvect(k)\in [0,\infty)^{\Jset}$ to $\mu_j(n_j,k)=\gamma_j(k) \mu_j(n_j)$.
 We use the notation introduced in \prettyref{sect:RS-JN-random-env-skip} and
$\avvect(k)$ and $\beta(k)$ as defined in
\eqref{eq:RS-general-availability-from-k}
and
\eqref{eq:RS-total-network-input-factor-from-k},
and take $\JsetB{\srfactorvect(k)}$ and $\JsetW{\srfactorvect(k)}$ as in  \prettyref{def:RS-JBA}.

We have the usual dynamics of the environment process $\Y$ with $\myV=(\nu(k,m): k,m\in K)$ and
$\envmx_j=(\envmx_j(k,m): k,m\in K), j\in \Jset$.
For the general rerouting regimes $\routmx^{(\avvect(k))}, k\in K,$  with
$\avvect(k)$ with $\alpha_0(k)=1$ and $\avvect(k)\in [0,1]^{\Jset_0}$,
 we only require the properties described in 
\prettyref{cor:RS-JN-GB-alpha-3} and have the following generalization in the spirit of Zhu's main theorem
 \cite{zhu:94}[p. 12],
without specifying explicitly the control regimes for rerouting. Note, that our
environment process is not Markov because of the two-way interaction of environment and service
process, while   Zhu's  theorem requires the environment to be Markov for its own.
\begin{cor}
\label{cor:GS-BOUNDARY-general}
The queue lengths-environment process $\Z=(\X,\Y) = (Z(t):t\geq 0) = ((X(t),Y(t)):t\geq 0)$
is a homogeneous Markov process  on state space
$E:= \N_0^{\Jset}\times K$  with generator
$Q^\Z = (q^\Z((\mathbf{n},k),(\mathbf{n}^{'},k')):
(\mathbf{n},k),(\mathbf{n}^{'},k')\in E)$.
The strict positive transition rates of $Q^\Z$, are for $(\mathbf{n},k)=((n_1,\dots,n_J),k)\in \mathbb{N}_{0}^{\Jset}\times K$
\begin{align*}
q((\nvect,k),(\nvect+\evect_{i},k)) & = \tnfactor(k)\lambda \routmx^{(\avvect(k))}(0,i), & i\in \Jset\,,
\nonumber \\
q((\nvect,k),(\nvect-\evect_{j}+\evect_{i},k)) &
= 1_{[n_{j}>0]}
\srfactor_j(k)\mu_{j}(n_{j})\routmx^{(\avvect(k))}(j,i),& j,i\in \Jset,~ i\neq j\,, \nonumber \\
q((\nvect,k),(\nvect-\evect_{j},m)) &
 = 1_{[n_{j}>0]}\srfactor_j(k)\mu_{j}(n_{j})
\routmx^{(\avvect(k))}(j,0)R_j(k,m),& j\in\Jset,\nonumber \\
q((\nvect,k),(\nvect,m)) & = \nu(k,m),& m\in K\,.
\end{align*}
Assume that the rerouting regimes $r^{(\avvect(k))}, ~k\in K,$ have invariant measures\\
$y(k)=(\alpha_j(k)\cdot  \eta_j:j\in \Jset_0)$.

Assume $\Z$ to be ergodic and assume further
that the pure Jackson network process $\X$ without environment is ergodic with stationary
and limiting distribution $\xi$ on $ \N_0^{\Jset}$  from \eqref{eq:RS-jackson-steadystate}
\begin{equation*}                        
\xi(\nvect)=\xi(n_1,\dots,n_J)
= \prod_{j=1}^{J} \prod_{k=1}^{n_j} \frac{\eta_j}{\mu_j(k)} C(j)^{-1},
\quad \nvect\in \N_0^{\Jset}\,.
\end{equation*}
Then the queue lengths-environment process $\Z$  has the unique
 steady state distribution $\pi=(\pi(\mathbf{n},k):\mathbf{n}\in\mathbb{N}_{0}^{\Jset},k\in K)$
of product form:
\[
\pi(\mathbf{n},k)=\xi(\mathbf{n})\theta(k),\quad \mathbf{n}\in\mathbb{N}_{0}^{\Jset},k\in K\,,
\]
where  $\theta$ is the unique stochastic solution of the following {\em reduced generator
equation} \\$\theta\cdot{Q}_{red}=0$ with

\begin{equation}\label{eq:RS-Q-red-general}
{Q}_{red} := {\left[\myV
+ \sum_{j\in\Jset} \eta_j I_{(\srfactor_j  \bullet \routmx^{(\avvect(\cdot))}(j,0))}(\envmx_j - I)\right]}\,.
\end{equation}

Here we denote for $j\in \Jset$ the real valued functions
$\srfactor_j$ and
$\routmx^{(\avvect(\cdot))}(j,0)$ on $K$,
and by\\
$\srfactor_j \bullet \routmx^{(\avvect(\cdot))}(j,0)$ their point wise multiplication, the result of which we
interprete as vector to obtain the diagonal matrix $I_{(\srfactor_j \bullet \routmx^{(\avvect(\cdot))}(j,0))}$.
\end{cor}

\subsection*{Note on figures}\label{sect:RS-public-domain-notice}
To the extent possible under law, the author(s) have dedicated all copyright and related and neighboring rights to \prettyref{fig:RS-routing-skipping}, \prettyref{fig:RS-jackson-skipping-degrading} and \prettyref{fig:RS-jackson-skipping-environement} to the public domain worldwide. 

You can find a copy of CC0 Public Domain Dedication on\\ \url{http://creativecommons.org/publicdomain/zero/1.0/}. 
    
\bibliographystyle{alpha}
\newcommand{\etalchar}[1]{$^{#1}$}


\end{document}